\documentclass[
	english
]{scrartcl}

\usepackage[T1]{fontenc}
\usepackage[utf8]{inputenc}

\usepackage{amsmath,amsfonts,amsthm,amssymb}
\usepackage[colorlinks,citecolor=blue]{hyperref}
\usepackage{doi,natbib}
\usepackage{cancel}
\usepackage{relsize}

\usepackage{algorithm,algpseudocode}

\usepackage{changes}

\usepackage{marginnote}
\usepackage[shortlabels]{enumitem}

\usepackage{enumitem}
\setlist[enumerate]{label=(\alph*)}

\usepackage[figure]{hypcap}
\usepackage{graphicx}
\usepackage{subcaption} 

\usepackage{preamble}

\usepackage{marginnote}

\usepackage{tikz}
\usetikzlibrary{arrows}

\tikzset{
    punkt/.style={
           rectangle,
           draw=white, very thick,
           text width=10em,
           minimum height=1.5em,
           text centered}
}

\makeatletter
\long\def\@firstoffiveparen#1#2#3#4#5{\textup{\tagform@{#1}}}
\def\eqref@nolink#1{\textup{\tagform@{\ref*{#1}}}}
\def\eqref@link#1{%
\Hy@safe@activestrue
\expandafter\@setref\csname r@#1\endcsname\@firstoffiveparen{#1}%
\Hy@safe@activesfalse
}
\protected\def\eqref{\@ifstar\eqref@nolink\eqref@link}
\makeatother

\usepackage[capitalise,nameinlink]{cleveref}

\crefname{figure}{Figure}{Figures}
\crefname{table}{Table}{Tables}

\numberwithin{table}{section}    
\numberwithin{figure}{section}  
\numberwithin{algorithm}{section}  

\allowdisplaybreaks

\numberwithin{equation}{section}

\newcommand\norm[1]{\Vert#1\Vert}
\newcommand\abs[1]{\left\vert#1\right\vert}

\newcommand\N{\mathbb{N}}
\newcommand\R{\mathbb{R}}

\newcommand{\intr}{\operatorname{int}}

\newcommand{\dist}{\operatorname{dist}}
\newcommand{\conv}{\operatorname{conv}}

\newcommand{\dom}{\operatorname{dom}}

\newcommand{\epi}{\operatorname{epi}}
\newcommand{\card}{\operatorname{card}}

\DeclareMathOperator*{\argmin}{\operatorname{arg\,min}}
\DeclareMathOperator*{\argmax}{\operatorname{arg\,max}}

\newcommand{\cdiff}{\partial^{\textup{c}}}
\newcommand{\gdiff}{\partial^{\square}}
\newcommand{\tto}{\rightrightarrows}
\newcommand{\MM}{\overline{\mathcal M}}

\DeclareMathAlphabet{\mathpzc}{OT1}{pzc}{m}{it}
\newcommand\oo{\mathpzc{o}}

\newtheorem{theorem}{Theorem}[section]
\newtheorem{lemma}[theorem]{Lemma}
\newtheorem{proposition}[theorem]{Proposition}

\newtheorem{remark}[theorem]{Remark}
\newtheorem{definition}[theorem]{Definition}
\newtheorem{example}[theorem]{Example}

\begin{document}

\title{
	Asymptotic regularity for Lipschitzian
	nonlinear optimization problems with applications to
	complementarity-constrained and bilevel programming
}

\author{%
	Patrick Mehlitz%
	\footnote{%
		Brandenburgische Technische Universit\"at Cottbus--Senftenberg,
		Institute of Mathematics,
		03046 Cottbus,
		Germany,
		\email{mehlitz@b-tu.de},
		\url{https://www.b-tu.de/fg-optimale-steuerung/team/dr-patrick-mehlitz},
		ORCID: 0000-0002-9355-850X%
		}
	}


\dedication{
	Dedicated to Stephan Dempe on the occasion of his 65th birthday.
}

\maketitle

\begin{abstract}
	 Asymptotic stationarity and regularity conditions turned out to be quite useful
	 to study the qualitative properties of numerical solution methods for
	 standard nonlinear and complementarity-constrained programs. 
	 In this paper, we first extend these notions to nonlinear optimization problems 
	 with nonsmooth but Lipschitzian data functions in order to find reasonable notions
	 of asymptotic stationarity and regularity in terms of Clarke's and Mordukhovich's
	 subdifferential construction. Particularly, we compare the associated novel asymptotic
	 constraint qualifications with already existing ones.
	 The second part of the paper presents two applications of the obtained theory.
	 On the one hand, we specify our findings for complementarity-constrained optimization 
	 problems and recover recent results from the literature which demonstrates
	 the power of the approach. Furthermore, we hint
	 at potential extensions to or- and vanishing-constrained optimization.
	 On the other hand, we demonstrate the usefulness of asymptotic regularity in the context of
	 bilevel optimization. More precisely, we justify a well-known stationarity system for
	 affinely constrained bilevel optimization problems in a novel way. Afterwards,
	 we suggest a solution algorithm for this class of bilevel optimization problems which
	 combines a penalty method with ideas from DC-programming. After a brief convergence
	 analysis, we present results of some numerical experiments.
\end{abstract}


\begin{msc}	
	49J52, 65K10, 90C30, 90C33
\end{msc}

\section{Introduction}\label{sec:introduction}

During the last decade, \emph{asymptotic} (sometimes referred to as sequential) notions of stationarity
and regularity have been developed for standard nonlinear optimization problems, 
see e.g.\ \cite{AndreaniFazzioSchuverdtSecchin2019,AndreaniHaeserMartinez2011,
AndreaniMartinezRamosSilva2016,AndreaniMartinezRamosSilva2018,AndreaniMartinezSvaiter2010},
complementarity-constrained programs, see \cite{AndreaniHaeserSecchinSilva2019,Ramos2019},
cardinality-constrained programs, see \cite{KanzowRaharjaSchwartz2021,KrulikovskiRibeiroSachine2020},
and programs in abstract spaces, see 
\cite{AndreaniHaeserViana2020,BoergensKanzowMehlitzWachsmuth2019}.
Extensions to nonsmooth optimization problems have been discussed recently in
\cite{HelouSantosSimoes2020}, based on Goldstein's $\varepsilon$-subdifferential, and in
\cite{Mehlitz2020b}, where the tools of limiting variational analysis have been exploited.
The huge interest in these concepts is based on their significant relevance for the investigation
of convergence properties associated with solution algorithms tailored for the aforementioned 
problem classes. More precisely, some numerical methods naturally produce \emph{asymptotically stationary}
points so the question arises which type of condition is necessary to hold at the limit in order
to guarantee its stationarity in the classical sense. 
The resulting \emph{asymptotic regularity} conditions
have been shown to serve as comparatively weak constraint qualifications for a bunch of problem
classes in mathematical programming.

In this paper, we apply the concepts of asymptotic stationarity 
and regularity to nonlinear optimization problems of the form
\begin{equation}\label{eq:Lipschitzian_program}\tag{P}
	\min\{\varphi_0(z)\,|\,\varphi_i(z)\leq 0\,(i\in I),\,\varphi_i(z)=0\,(i\in J)\},
\end{equation}
where the data functions $\varphi_0,\ldots,\varphi_{p+q}\colon\R^n\to\overline{\R}$ 
are assumed to be locally Lipschitz continuous but not necessarily smooth in 
a neighborhood of a given reference point.
Here, we use $I:=\{1,\ldots,p\}$ and $J:=\{p+1,\ldots,p+q\}$. 
For that purpose, we will exploit the subdifferential concepts of Clarke and Mordukhovich,
see e.g.\ \cite{Clarke1983,Mordukhovich2006}, respectively, since these generalized derivatives
are outer semicontinuous (in the sense of set-valued mappings) 
by construction which will be beneficial for our theoretical
investigations. In \cite[Section~1]{HelouSantosSimoes2020}, the authors point out that this
approach may have the disadvantage that numerical methods associated with
\eqref{eq:Lipschitzian_program} may not compute stationary points in this new sense.
However, as we will see in \cref{sec:MPCCs}, our results recover recently introduced notions of
asymptotic stationarity and regularity for so-called \emph{mathematical programs with complementarity
constraints} (MPCCs for short),
see \cite{LuoPangRalph1996,OutrataKocvaraZowe1998},
which have been shown to be useful in numerical practice,
see \cite{AndreaniHaeserSecchinSilva2019,Ramos2019}. 
We note that our model program
\eqref{eq:Lipschitzian_program} covers other prominent classes from disjunctive programming like
so-called or- and vanishing-constrained programs, see 
\cite{AchtzigerKanzow2008,HoheiselKanzow2007,HoheiselPablosPooladianSchwartzSteverango2020,Mehlitz2019,Mehlitz2020}
and \cref{rem:or_constrained_programming}, so that our findings are likely to possess reasonable
extensions to these models as well.
Based on our new notions of asymptotic stationarity, we introduce three novel sequential constraint
qualifications which guarantee that asymptotically stationary points of
\eqref{eq:Lipschitzian_program} are stationary in Clarke's or Mordukhovich's sense, i.e., that
these points satisfy Karush--Kuhn--Tucker-type stationarity conditions based on Clarke's or
Mordukhovich's subdifferential. Afterwards, we study the relationship between these new
regularity conditions and already available constraint qualifications from nonsmooth programming.
Particularly, we address connections to a nonsmooth version of the so-called
\emph{relaxed constant positive linear dependence constraint qualification} (RCPLD), which has been
introduced for smooth standard nonlinear optimization problems in
\cite{AndreaniHaeserSchuverdtSilva2012} and extended to nonsmooth programs in \cite{XuYe2020} quite
recently. 

As already mentioned, we apply our quite general findings regarding the abstract model
\eqref{eq:Lipschitzian_program} to MPCCs in \cref{sec:MPCCs} in order to underline the particular value
as well as the applicability of these results. 
In \cref{sec:bilevel}, we demonstrate the power of asymptotic stationarity and regularity
in the context of bilevel optimization with affine data in the upper level constraints as
well as in the overall lower level. It is well known that bilevel optimization problems are
notoriously difficult due to their inherent irregularity, nonsmoothness, and nonconvexity, while
being a major topic in the focus of many researchers because of their overwhelming
practical relevance with respect to (w.r.t.) the modeling of real-world applications from finance, economics, or
natural and engineering sciences, 
see \cite{Dempe2002,Dempe2020,DempeKalashnikovPerezValdesKalashnykova2015} 
for an introduction to bilevel optimization and a comprehensive literature review.
Here, we make use of the so-called optimal value transformation of bilevel optimization problems
in order to transfer the program of interest into the form \eqref{eq:Lipschitzian_program}.
Noting that the latter is inherently asymptotically regular in the investigated setting, 
we are in position to state necessary optimality conditions without any further assumptions
or the use of partial penalization arguments. Afterwards, we suggest a solution method for
the considered problem class which penalizes the constraint comprising the optimal value function
and uses methods from DC-programming, where DC abbreviates \emph{difference of convex functions},
see \cite{HorstThoai1999,ThiDinh2018} for an overview,
in order to solve the subproblems. As we will see, this approach is computationally reasonable since
it exploits only pointwise evaluations of function values and subgradients of the 
optimal value function which can be easily computed while the overall optimal value function 
may remain an implicitly given object. Furthermore, we can apply our abstract theory from
\cref{sec:asymptotic_concepts} in order to demonstrate that our method computes stationary
points of the bilevel optimization problem of interest. Some numerical results visualize the
computational performance of the method.

The remaining parts of this paper are organized as follows.
In \cref{sec:notation}, we summarize the notation used in this manuscript, recall some fundamental
notions from nonsmooth differentiation, and present some preliminary results.
\Cref{sec:asymptotic_concepts} is dedicated to the formal introduction of asymptotic stationarity and
regularity notions which address \eqref{eq:Lipschitzian_program}. Noting that we proceed in a 
fairly standard way here, many nearby proofs are left out for the purpose of brevity. 
Instead, we focus on the relationship between the new notions of asymptotic regularity and already
available constraint qualifications from nonsmooth optimization.
These results are applied to MPCCs in \cref{sec:MPCCs}. As we will see, we precisely recover already
available theory from the literature which has been obtained using a standard local decomposition
approach. In \cref{sec:bilevel}, bilevel optimization problems of special structure are studied in the
light of asymptotic stationarity and regularity. Particularly, we state a stationarity condition
which holds at all local minimizers, formulate a numerical method which is capable of finding stationary
points in this sense, and present some associated numerical results. The paper closes with some
concluding remarks in \cref{sec:conclusions}.

\section{Notation and preliminaries}\label{sec:notation}

The general notation in this paper is standard.
We use $\overline{\R}:=\R\cup\{-\infty,\infty\}$ in order to denote the extended real line.
The space $\R^n$ is equipped with the Euclidean norm $\norm{\cdot}$. For $z\in\R^n$ and $\varepsilon>0$,
$\mathbb B_\varepsilon(z):=\{y\in\R^n\,|\,\norm{y-z}\leq\varepsilon\}$ represents the closed
$\varepsilon$-ball around $z$. 
The distance function $\dist_K\colon\R^n\to\R$ of a closed, convex set $K\subset\R^n$ is given by
$\dist_K(z):=\inf\{\norm{y-z}\,|\,y\in K\}$ for each $z\in\R^n$. Moreover, 
$\Pi_K\colon\R^n\to\R^n$ denotes the projection map associated with $K$.
Whenever $\phi\colon\R^n\to\R$ is smooth at some point $\bar z\in\R^n$, $\nabla\phi(\bar z)\in\R^n$
is used to denote the gradient of $\phi$ at $\bar z$. For a function $\Phi\colon\R^n\to\R^m$ and
a vector $y\in\R^m$, the mapping $\langle y,\Phi\rangle\colon\R^n\to\R$ is given by
$\langle y,\Phi\rangle(z):=y^\top\Phi(z)$ for each $z\in\R^n$.
For brevity of notation, a tuple $(z_1,\ldots,z_n)$ of real numbers $z_1,\ldots,z_n\in\R$ will
be identified with a vector from $\R^n$ which possesses the components $z_1,\ldots,z_n$.
For finite index sets $I_1$ and $I_2$ as well as families $(a_i)_{i\in I_1},(b_i)_{i\in I_2}\subset\R^n$,
we call the pair of families $\bigl((a_i)_{i\in I_1},(b_i)_{i\in I_2}\bigr)$ positive linearly
dependent whenever there exist scalars $\alpha_i\geq 0$ ($i\in I_1$) and $\beta_i$ ($i\in I_2$), not
all vanishing simultaneously, such that $\sum_{i\in I_1}\alpha_i\,a_i+\sum_{i\in I_2}\beta_i\,b_i=0$.
For a set $A\subset\R^n$ and some point $z\in\R^n$, we use $A+z:=\{a+z\,|\,a\in A\}=:z+A$ for
brevity.
Finally, for a set-valued mapping $\Gamma\colon\R^n\tto\R^m$ and some point $\bar z\in\R^n$, we use
\[
	\limsup\limits_{z\to\bar z}\Gamma(z)
	:=
	\left\{	
		\xi\in\R^m\,\middle|\,
			\begin{aligned}
				&\exists\{z^k\}_{k\in\N}\subset\R^n,\,\exists\{\xi^k\}_{k\in\N}\subset\R^m\colon\\
				&\qquad z^k\to\bar z,\,\xi^k\to\xi,\,\xi^k\in\Gamma(z^k)\,\forall k\in\N
			\end{aligned}
	\right\}
\]
in order to denote the outer (or Painlev\'{e}--Kuratowski) limit of $\Gamma$ at $\bar z$.
Note that $\Gamma(\bar z)\subset\limsup_{z\to\bar z}\Gamma(z)$ holds always true, and that the
outer limit is always closed. In case where $\limsup_{z\to\bar z}\Gamma(z)\subset\Gamma(\bar z)$
is valid, $\Gamma$ is said to be outer semicontinuous at $\bar z$.

\subsection{Variational analysis}\label{sec:variational_analysis}

Subsequently, we recall some notions from nonsmooth analysis and generalized differentiation
which can be found, e.g., in \cite{Clarke1983,Mordukhovich2006,RockafellarWets1998}.

For a closed set $A\subset\R^n$ and some point $\bar z\in A$, the regular and limiting normal
cone to $A$ at $\bar z$ are given, respectively, by means of
\begin{align*}
	\widehat{\mathcal N}_A(\bar z)
	&:=
	\{\xi\in\R^n\,|\,\forall z\in A\colon\,\xi^\top(z-\bar z)\leq\oo(\norm{z-\bar z})\},\\
	\mathcal N_A(\bar z)
	&:=
	\limsup\limits_{z\to\bar z,\,z\in A}\widehat{\mathcal N}_A(z).
\end{align*}
We note that in situations where $A$ is convex, these cones coincide with 
the standard normal cone from convex analysis, i.e., we find
\[
	\widehat{\mathcal N}_A(\bar z)
	=
	\mathcal N_A(\bar z)
	=
	\{\xi\in\R^n\,|\,\forall z\in A\colon\,\xi^\top(z-\bar z)\leq 0\}
\]
in this situation.

For a lower semicontinuous function $\psi\colon\R^n\to\overline{\R}$, 
$\dom\psi:=\{z\in\R^n\,|\,|\psi(z)|<\infty\}$ and
$\epi\psi:=\{(z,\alpha)\in\R^n\times\R\,|\,\psi(z)\leq\alpha\}$ denote its domain
and epigraph, respectively. 
Note that these sets are closed.
Let us fix a point $\bar z\in\dom\psi$ where $\psi$ is locally Lipschitz continuous. Then the set
\[
	\widehat{\partial}\psi(\bar z)
	:=
	\left\{
		\xi\in\R^n\,\middle|\,(\xi,-1)\in\widehat{\mathcal N}_{\epi\psi}(\bar z,\psi(\bar z))
	\right\}
\]
is called the regular (or Fr\'{e}chet) subdifferential of $\psi$ at $\bar z$. 
Furthermore, we refer to
\[
	\partial\psi(\bar z)
	:=
	\left\{
		\xi\in\R^n\,\middle|\,(\xi,-1)\in \mathcal N _{\epi\psi}(\bar z,\psi(\bar z))
	\right\}
\]
as the limiting (or Mordukhovich) subdifferential of $\psi$ at $\bar z$. Finally, 
\[
	\cdiff\psi(\bar z)
	:=\conv\partial\psi(\bar z),
\]
i.e., the convex hull of $\partial\psi(\bar z)$, is referred to as the Clarke 
(or convexified) subdifferential of $\psi$ at $\bar z$. By definition, we have 
$\widehat{\partial}\psi(\bar z)\subset\partial\psi(\bar z)\subset\cdiff\psi(\bar z)$,
and whenever $\psi$ is convex, all these sets coincide with the subdifferential of $\psi$ in the
sense of convex analysis. We note that the regular and limiting subdifferential are positive
homogeneous while Clarke's subdifferential is homogeneous.

Some properties of Mordukhovich's and Clarke's subdifferential, which we will point out below, 
are of essential importance in this paper.
Therefore, let us assume again that $\psi$ is locally Lipschitz continuous at $\bar z$. Then
$\gdiff\psi(\bar z)$ is nonempty where $\gdiff$ is a representative of the operators $\partial$ and
$\cdiff$ (here and in the course of this paper). Furthermore, the set-valued
mapping $z\mapsto\gdiff\psi(z)$
is locally bounded at $\bar z$, i.e., there are a neighborhood $U$ of $\bar z$ and a bounded set
$B\subset\R^n$ such that $\gdiff\psi(z)\subset B$ is valid for all $z\in U$. 
Additionally, the set-valued mapping $z\mapsto\gdiff\psi(z)$ is outer semicontinuous 
at $\bar z$.
This property may be also referred to as \emph{robustness} of the subdifferential $\gdiff$.

Below, we present a calculus rule for the subdifferential of minimum functions.
Note that we do not only provide upper estimates but precise formulas here.
\begin{lemma}\label{lem:composed_subdifferential_of_minimum_function}
	Let $f_1,f_2\colon\R^n\to\R$ be continuously differentiable functions and consider
	the function $\varphi\colon\R^n\to\R$ given by
	$\varphi(z):=\min(f_1(z),f_2(z))$ for all $z\in\R^n$.
	For each point $\bar z\in\R^n$, the following formulas hold:
	\begin{align*}
		\partial\varphi(\bar z)
		&=
		\begin{cases}
			\{\nabla f_1(\bar z)\}	& f_1(\bar z)<f_2(\bar z),\\
			\{\nabla f_1(\bar z),\nabla f_2(\bar z)\}	& f_1(\bar z)=f_2(\bar z),\\
			\{\nabla f_2(\bar z)\}	& f_1(\bar z)>f_2(\bar z),
		\end{cases}
		\\
		\cdiff\varphi(\bar z)
		&=
		\begin{cases}
			\{\nabla f_1(\bar z)\}	& f_1(\bar z)<f_2(\bar z),\\
			\conv\{\nabla f_1(\bar z),\nabla f_2(\bar z)\}	& f_1(\bar z)=f_2(\bar z),\\
			\{\nabla f_2(\bar z)\}	& f_1(\bar z)>f_2(\bar z),
		\end{cases}
		\\
		\partial(-\varphi)(\bar z)
		&=\cdiff(-\varphi)(\bar z)
		=
		\begin{cases}
			\{-\nabla f_1(\bar z)\}	& f_1(\bar z)<f_2(\bar z),\\
			\conv\{-\nabla f_1(\bar z),-\nabla f_2(\bar z)\}	& f_1(\bar z)=f_2(\bar z),\\
			\{-\nabla f_2(\bar z)\}	& f_1(\bar z)>f_2(\bar z).
		\end{cases}
	\end{align*}
\end{lemma}
\begin{proof}
	Exploiting $-\varphi(z)=\max(-f_1(z),-f_2(z))$ which holds for each $z\in\R^n$, 
	the formulas for $\cdiff(-\varphi)(\bar z)$  and $\partial(-\varphi)(\bar z)$
	follow from the maximum rules
	\cite[Proposition~2.3.12]{Clarke1983} and
	\cite[Theorem~3.46(ii)]{Mordukhovich2006}, respectively, 
	while observing that $f_1$ and $f_2$ are
	continuously differentiable. 
	Thus, the formula for $\cdiff\varphi(\bar z)$ is obtained 
	from homogeneity of Clarke's subdifferential. 
	
	It remains to prove the formula for $\partial\varphi(\bar z)$.
	The inclusion $\subset$ is shown in \cite[Proposition~1.113]{Mordukhovich2006}.
	In order to validate the converse inclusion, we employ a distinction of cases.
	If $f_1(\bar z)<f_2(\bar z)$ holds, then we find $\varphi(z)=f_1(z)$ locally
	around $\bar z$, and by continuous differentiability of $f_1$, 
	$\partial\varphi(\bar z)=\{\nabla f_1(\bar z)\}$ follows. Similarly, we can address
	the situation $f_1(\bar z)>f_2(\bar z)$. Finally, let us investigate the case
	$f_1(\bar z)=f_2(\bar z)$. We will show $\nabla f_1(\bar z)\in\partial\varphi(\bar z)$.
	Similarly, one obtains $\nabla f_2(\bar z)\in\partial\varphi(\bar z)$.
	Suppose that there is a sequence $\{z^k\}_{k\in\N}\subset\R^n$
	such that $z^k\to\bar z$ and $f_1(z^k)<f_2(z^k)$ for all $k\in\N$. 
	By continuity of $f_1$ and $f_2$ as well as continuous differentiability of $f_1$,
	we find $\widehat{\partial}\varphi(z^k)=\{\nabla f_1(z^k)\}$, and taking the 
	limit $k\to\infty$ yields $\nabla f_1(\bar z)\in\partial\varphi(\bar z)$.
	If such a sequence $\{z^k\}_{k\in\N}$ does not exist, 
	we find a neighborhood $U$ of $\bar z$ such that
	$f_1(z)\geq f_2(z)$ holds for all $z\in U$. On the one hand, this shows
	$\varphi(z)=f_2(z)$ for all $z\in U$ and, thus, $\partial\varphi(\bar z)=\{\nabla f_2(\bar z)\}$.
	On the other hand, due to $f_1(\bar z)=f_2(\bar z)$, $\bar z$ is a local minimizer
	of $f_1-f_2$ which yields $\nabla f_1(\bar z)-\nabla f_2(\bar z)=0$. Summing this
	up, we have shown $\nabla f_1(\bar z)\in\partial \varphi(\bar z)$.	
\end{proof}

\subsection{Sequential stationarity for optimization problems with Lipschitzian geometric constraints}
	\label{sec:geometric_constraints}

In this section, we investigate the optimization problem
\begin{equation}\label{eq:geometric_constraints}\tag{Q}
	\min\{f(z)\,|\,F(z)\in K\}
\end{equation}
where $f\colon\R^n\to\overline{\R}$ and $F\colon\R^n\to\overline{\R}^m$ are given functions
and $K\subset\R^m$ is a convex polyhedral set. 
Let $\mathcal F:=\{z\in\R^n\,|\,F(z)\in K\}$ be the feasible set of \eqref{eq:geometric_constraints}.
For later use, we are going to characterize local minimizers of \eqref{eq:geometric_constraints}
with the aid of sequential stationarity conditions which are based on the limiting
subdifferential. Related results can be found in
\cite{Mehlitz2020b,Ramos2019}.
\begin{proposition}\label{prop:sequential_stationarity_geometric_constraints}
	Let $\bar z\in\mathcal F$ be a local minimizer of \eqref{eq:geometric_constraints}.
	Furthermore, assume that $f$ and $F$ are Lipschitz continuous around $\bar z$.
	Then there exist sequences 
		$\{z^k\}_{k\in\N},\{\varepsilon^k\}_{k\in\N}\subset\R^n$ and 
		$\{\lambda^k\}_{k\in\N}\subset\R^m$
		such that $z^k\to\bar z$, $\varepsilon^k\to 0$, and
			\begin{subequations}\label{eq:seq_stat_geo}
				\begin{align}
					\label{eq:seq_stat_geo_der}
						&\forall k\in\N\colon\quad
						\varepsilon^k\in\partial f(z^k)+\partial\langle\lambda^k,F\rangle(z^k),\\
					\label{eq:seq_stat_geo_mult}
						&\forall k\in\N\colon\quad
							\lambda^k\in\widehat{\mathcal N}_K(F(\bar z)).
				\end{align}
			\end{subequations}
\end{proposition} 
\begin{proof}
	We choose $\delta>0$ such that $f$ and $F$ are locally Lipschitz
	continuous on $\mathbb B_{2\delta}(\bar z)$ while $f(z)\geq f(\bar z)$ holds
	for all $z\in \mathcal F\cap\mathbb B_\delta(\bar z)$.
	For each $k\in\N$, we consider the penalized problem
	\begin{equation}
		\label{eq:penalized_problem}\tag{Q$(k)$}
			\min\{
				f(z)+\tfrac k2\dist^2_K(F(z))+\tfrac12\norm{z-\bar z}^2
				\,|\,
				z\in\mathbb B_\delta(\bar z)
				\}.
	\end{equation}
	Noting that the objective function of \eqref{eq:penalized_problem} is locally Lipschitz continuous
	on the feasible set, this program possesses a global minimizer $z^k\in\mathbb B_\delta(\bar z)$ for
	each $k\in\N$. 
	Since $\{z^k\}_{k\in\N}$ is bounded, we may assume without loss of generality (w.l.o.g.) 
	that there is some 
	$\tilde z\in\mathbb B_\delta(\bar z)$ such that $z^k\to\tilde z$ holds as $k\to\infty$.
	By feasibility of $\bar z$, we infer
	\[
		\forall k\in\N\colon\quad 
		f(z^k)+\tfrac k2\dist^2_K(F(z^k))+\tfrac12\norm{z^k-\bar z}^2\leq f(\bar z).
	\]
	Since $f$ is continuous on $\mathbb B_\delta(\bar z)$, 
	$\{f(z^k)\}_{k\in\N}$ is a bounded sequence. Thus,
	we find a constant $C>0$ which satisfies $\tfrac k2\dist_K^2(F(z^k))\leq C$.
	Consequently, we have $\dist_K^2(F(z^k))\to 0$ as $k\to\infty$.
	Exploiting the continuity of the distance function as well as $F$, this yields 
	$F(\tilde z)\in K$, i.e., $\tilde z\in \mathcal F\cap \mathbb B_\delta(\bar z)$.
	By choice of $\delta$, this leads to
	\begin{align*}
		f(\tilde z)
		&\leq
		f(\tilde z)+\tfrac12\norm{\tilde z-\bar z}^2\\
		&\leq
		\limsup\limits_{k\to\infty}
			\left(f(z^k)+\tfrac k2\dist_K^2(F(z^k))+\tfrac12\norm{z^k-\bar z}^2\right)\\
		&\leq
		f(\bar z)
		\leq
		f(\tilde z),
	\end{align*}
	and, thus, we have $\tilde z=\bar z$. 
	Additionally, we may assume w.l.o.g.\ that $\{z^k\}_{k\in\N}$ belongs to the interior of 
	$\mathbb B_\delta(\bar z)$.
	
	Next, we define functions
	$f_1,f_2\colon\R^n\to\overline{\R}$ by means of
	\[
		\forall z\in\R^n\colon\quad
		f_1(z):=\tfrac12\dist^2_K(F(z)),\qquad f_2(z):=\tfrac12\norm{z-\bar z}^2.
	\]
	Recalling that $z^k$ is a minimizer of \eqref{eq:penalized_problem}, Fermat's rule,
	see \cite[Proposition~1.114]{Mordukhovich2006}, guarantees
	$0\in\widehat{\partial}(f+k\,f_1+f_2)(z^k)$ for each $k\in\N$. Since $f_2$ is continuously
	differentiable with gradient $\nabla f_2(z)=z-\bar z$ for arbitrary $z\in\R^n$, 
	$\bar z-z^k\in\widehat{\partial}(f+k\,f_1)(z^k)$ follows from
	\cite[Proposition~1.107]{Mordukhovich2006}.
	Next, we apply the sum rule for the limiting subdifferential, 
	see \cite[Theorem~3.36]{Mordukhovich2006}, in order to find that
	$\bar z-z^k\in\partial f(z^k)+k\,\partial f_1(z^k)$
	holds for each $k\in\N$. Due to $f_1=\tfrac12\dist_K^2\circ F$, continuous differentiability
	of the squared distance function to a convex set, and Lipschitzianity of $F$
	at $z^k$, we find
	\[
		\partial f_1(z^k)
		=
		\partial \langle F(z^k)-\Pi_K(F(z^k)),F\rangle(z^k)
	\]
	from the subdifferential chain rule \cite[Corollary~3.43]{Mordukhovich2006}. 
	Setting $\varepsilon^k:=\bar z-z^k$
	and $\lambda^k:=k\bigl(F(z^k)-\Pi_K(F(z^k))\bigr)$ for each $k\in\N$, we find
	$\varepsilon^k\to 0$ and \eqref{eq:seq_stat_geo_der}.
	By construction, we have $z^k\to\bar z$
	and $\lambda^k\in\widehat{\mathcal N}_K(\Pi_K(F(z^k)))$ for
	each $k\in\N$, see \cite[Section~6.E]{RockafellarWets1998}.
	Due to $F(z^k)\to F(\bar z)$ and the polyhedrality of $K$, we can 
	apply \cite[Lemma~2.1]{Mehlitz2019b} in order to find \eqref{eq:seq_stat_geo_mult}
	for large enough $k\in\N$.
	This completes the proof.
\end{proof}

Let us comment on the assertion of \cref{prop:sequential_stationarity_geometric_constraints}.
First, we would like to point the reader's attention to the fact that the appearing multiplier sequence
$\{\lambda^k\}_{k\in\N}$  does not need to be bounded.
If this would be the case, then one could simply take the limit along some convergent subsequence
in \eqref{eq:seq_stat_geo} in order to find some
$\lambda\in\R^m$ which satisfies 
\[
	0\in\partial f(\bar z)+\partial\langle\lambda,F\rangle(\bar z),
	\qquad
	\lambda\in\widehat{\mathcal N}_K(F(\bar z)).
\]
This follows by robustness of the limiting subdifferential, 
see \cite[Lemma~3.4]{Mehlitz2020b} as well, and convexity of $K$.
Note that due to the appearance of the limiting subdifferential, these conditions precisely correspond to the so-called 
Mordukhovich (or simply M-) stationarity conditions of
\eqref{eq:geometric_constraints} at $\bar z$
(recall that since $K$ is convex, the limiting and regular normal cone to this set coincide).
Clearly, there exist optimization problems whose
local minimizers are \emph{not} M-stationary so that it is completely reasonable that we are
not in position to show boundedness of $\{\lambda^k\}_{k\in\N}$ without additional regularity
in the proof of \cref{prop:sequential_stationarity_geometric_constraints}.

Next, we would like to point out that the conditions in
\cref{prop:sequential_stationarity_geometric_constraints}
precisely correspond to the so-called AM-stationarity conditions of \eqref{eq:geometric_constraints},
which were introduced in \cite[Definition~3.1]{Mehlitz2020b} for much more general problems.
This can be seen by employing the scalarization property of the so-called limiting coderivative,
see \cite[Theorem~1.90]{Mordukhovich2006}. 

Finally, let us mention that it is also possible to state the assertion of
\cref{prop:sequential_stationarity_geometric_constraints} in terms of the
regular subdifferential. However, one has to exploit the so-called \emph{fuzzy}
sum rule during the proof, see \cite[Theorem~2.33]{Mordukhovich2006}, since the regular subdifferential does not
obey a classical sum rule. Respecting this, one would have to
replace \eqref{eq:seq_stat_geo_der} by
\begin{equation}\label{eq:sequential_stationarity_Frechet}
	\forall k\in\N\colon\quad 
	\varepsilon^k\in\widehat{\partial}f(z^k_\textup{o})
	+
	\widehat{\partial}\langle\lambda^k,F\rangle(z^k_\textup{c})
\end{equation}
where $\{z^k_\textup{o}\}_{k\in\N},\{z^k_\textup{c}\}_{k\in\N}\subset\R^n$ are
sequences satisfying $z^k_\textup{o}\to\bar z$ and $z^k_\textup{c}\to\bar z$.
By definition of the limiting subdifferential, one can show that, in general, this condition 
is not stronger than the one postulated in \cref{prop:sequential_stationarity_geometric_constraints}
as long as both $f$ and $F$ are nonsmooth around $\bar z$.
Let us also note that when taking the limit in \eqref{eq:sequential_stationarity_Frechet}, one would
end up with a condition in terms of the limiting subdifferential anyway.
That is why we rely on the statement of \cref{prop:sequential_stationarity_geometric_constraints}
in the remainder of the paper.

\section{Asymptotic stationarity and regularity for Lipschitzian nonlinear programs}
	\label{sec:asymptotic_concepts}

In the past, several tools of generalized differentiation have been introduced which allow to
transfer the Karush--Kuhn--Tucker (KKT) theory for standard nonlinear programs with continuously
differentiable data functions to a nonsmooth framework.
Amongst others, let us mention the subdifferential constructions introduced by Clarke and Mordukhovich, 
see \cite{Clarke1983,Mordukhovich2006}, which enjoy (almost) full calculus and can be used
to derive KKT-type necessary optimality conditions for the model problem \eqref{eq:Lipschitzian_program},
see \cite[Section~5.6]{Vinter2000} and \cite[Section~5.1.3]{Mordukhovich2006}.
We will refer to these conditions as the systems of $\cdiff$- and $\partial$-stationarity, respectively.
In the literature, the nomenclatures of Clarke (or simply C-) and M-stationarity are also common, 
but we will avoid these terms here
for some reasons which will become clear in the course of the paper.
For the purpose of completeness, we start our investigations by stating a precise definition of 
$\cdiff$- and $\partial$-stationarity, respectively.

Throughout the section, let $\mathcal Z$ denote the feasible set of \eqref{eq:Lipschitzian_program}.
We implicitly assume  that whenever $\bar z\in\mathcal Z$ is
a fixed feasible point of \eqref{eq:Lipschitzian_program}, 
then the functions $\varphi_0,\ldots,\varphi_{p+q}$ are locally Lipschitz continuous in a neighborhood
of $\bar z$. Furthermore, we make use of the so-called index set associated with inequality constraints
active at $\bar z$ which is given by $I(\bar z):=\{i\in I\,|\,\varphi_i(\bar z)=0\}$.
Finally, recall that $\gdiff$ plays the role of the subdifferential operator
$\cdiff$ or $\partial$.

\begin{definition}\label{def:C_M_stationarity}
	A feasible point $\bar z\in\mathcal Z$ of \eqref{eq:Lipschitzian_program} is called
	$\gdiff$-stationary whenever there are multipliers $\lambda\in\R^{p+q}$ which satisfy
	the following conditions:
	\begin{subequations}\label{eq:C_M_stationarity}
		\begin{align}
			\label{eq:C_M_stationarity_der}
				&0\in\partial\varphi_0(\bar z)
				+
				\sum\limits_{i\in I}\lambda_i\,\gdiff\varphi_i(\bar z)
				+
				\sum\limits_{i\in J}\lambda_i\,
					\bigl(\gdiff\varphi_i(\bar z)\cup\gdiff(-\varphi_i)(\bar z)\bigr),\\
			\label{eq:C_M_stationarity_compl_slack}
				&\forall i\in I\colon\quad \min(\lambda_i,-\varphi_i(\bar z))=0,\\
			\label{eq:C_M_stationarity_equality_constraints}
				&\forall i\in J\colon\quad \lambda_i\geq 0.
		\end{align}
	\end{subequations}
\end{definition}

Although $\partial$-stationarity is sharper than $\cdiff$-stationarity, it might be
beneficial to work with Clarke's subdifferential in some situations since it is far easier
to compute, see \cite[Theorem~2.5.1]{Clarke1983}, 
and its homogeneity allows for an easier calculus, see e.g.\ \cref{sec:bilevel}
where this issue is of essential importance. Similar arguments justify the consideration
of constraint qualifications based on Clarke's subdifferential. Let us point out that in
both stationarity systems, the limiting subdifferential is used as the generalized derivative for
the objective function. At the first glance, this seems to be uncommon. However, since most of
the variational issues one has to face during the theoretical treatment of \eqref{eq:Lipschitzian_program}
are related to the structure of the feasible set, see e.g.\ \cref{sec:MPCCs,sec:bilevel}, 
this choice is reasonable and induces a system of
$\cdiff$-stationarity which is slightly sharper than the classical one from \cite{Vinter2000}.

Let us put the stationarity concepts from \cref{def:C_M_stationarity} into some context. 
For that purpose, let $\Phi\colon\R^n\to\R^{p+q}$ be the vector function whose components
are precisely $\varphi_1,\ldots,\varphi_{p+q}$.
We note that \eqref{eq:Lipschitzian_program}
is a particular instance of the problem \eqref{eq:geometric_constraints} discussed in 
\cref{sec:geometric_constraints} where we fix $f:=\varphi_0$, $F:=\Phi$, and
$K:=\{y\in\R^{p+q}\,|\,\forall i\in I\colon\,y_i\leq 0,\,\forall i\in J\colon\,y_i=0\}$. 
A simple evaluation of the regular normal cone to this particular set $K$ reveals that
\begin{subequations}\label{eq:seq_stat_implicit}
	\begin{align}	
	\label{eq:seq_stat_implicit_der}
	&0\in\partial\varphi_0(\bar z)+\gdiff\langle\tilde\lambda,\Phi\rangle(\bar z),\\
	\label{eq:seq_stat_implicit_mult}
	&\forall i\in I\colon\quad\min(\tilde\lambda_i,-\varphi_i(\bar z))=0
	\end{align}
\end{subequations}
for some multiplier $\tilde\lambda\in\R^{p+q}$ might be a reasonable candidate for the
$\gdiff$-stationarity system as well. The sum rule for Clarke's subdifferential 
as well as its homogeneity, see \cite[Section~2.3]{Clarke1983}, imply
\[
	\cdiff\langle\tilde\lambda,\Phi\rangle(\bar z)
	\subset
	\sum\limits_{i\in I\cup J}\tilde\lambda_i\,\cdiff\varphi_i(\bar z),	
\]
and due to
\[
	\gamma\,\cdiff\psi(z)
	=
	\begin{cases}
		\abs{\gamma}\,\cdiff\psi(z)&\gamma\geq 0,\\
		\abs{\gamma}\,\cdiff(-\psi)(z)&\gamma<0
	\end{cases}
\]
for each function $\psi\colon\R^n\to\overline{\R}$ which is locally Lipschitzian around $z\in\R^n$ and each constant
$\gamma\in\R$, the stationarity condition \eqref{eq:seq_stat_implicit} for
Clarke's subdifferential is slightly stronger than $\cdiff$-stationarity from
\cref{def:C_M_stationarity}. On the other hand, we find the inclusion
\[
	\partial\langle\tilde\lambda,\Phi\rangle(\bar z)
	\subset
	\sum\limits_{i\in I}\tilde\lambda_i\,\partial\varphi_i(\bar z)
	+
	\sum\limits_{i\in J}|\tilde\lambda_i|\bigl(\partial\varphi_i(\bar z)
	\cup\partial(-\varphi_i)(\bar z)\bigr)
\]
by the sum rule for the limiting subdifferential, see \cite[Theorem~3.36]{Mordukhovich2006}, 
and its positive homogeneity. Thus,
\eqref{eq:C_M_stationarity_der} from \cref{def:C_M_stationarity} for the limiting subdifferential
might be weaker than the condition
\eqref{eq:seq_stat_implicit_der} from above. However, the system 
\eqref{eq:C_M_stationarity} is stated in fully explicit way w.r.t.\ the subdifferentials of the
appearing constraint functions. It is, thus, reasonable to work with the generalized stationarity
notions from \cref{def:C_M_stationarity} and not with the potentially sharper
conditions from \eqref{eq:seq_stat_implicit}. 

Recently, the concept of \emph{asymptotic} stationarity has attracted lots of attention due to two
basic observations. First, some algorithms from optimization theory naturally produce a sequence
of iterates whose accumulation points satisfy such asymptotic stationarity conditions.
Second, asymptotic stationarity gives rise to the definition of very weak constraint qualifications.
We refer the interested reader to 
\cite{AndreaniHaeserMartinez2011,AndreaniHaeserSecchinSilva2019,AndreaniMartinezRamosSilva2016,AndreaniMartinezSvaiter2010,AndreaniMartinezRamosSilva2018,BoergensKanzowMehlitzWachsmuth2019,Mehlitz2020b,Ramos2019}
and the references therein for a validation.
Below, we present two natural extensions of asymptotic stationarity which apply to the Lipschitzian
optimization problem \eqref{eq:Lipschitzian_program} and are based on $\cdiff$- and
$\partial$-stationarity from \cref{def:C_M_stationarity}.

\begin{definition}\label{def:asymptotic_stationarity}
	A feasible point $\bar z\in\mathcal Z$ of \eqref{eq:Lipschitzian_program} is 
	called asymptotically $\gdiff$-stationary (A$\gdiff$-stationary for short)
	whenever there are sequences $\{z^k\}_{k\in\N},\{\varepsilon^k\}_{k\in\N}\subset\R^n$ and
	$\{\lambda^k\}_{k\in\N}\subset\R^{p+q}$ which satisfy
	\begin{subequations}\label{eq:asymptotic_stationarity}
		\begin{align}
			\label{eq:asymptotic_stationarity_Lagrangian}
				&\varepsilon^k
				\in 
				\partial\varphi_0(z^k)
				+
				\sum\limits_{i\in I}\lambda^k_i\,\gdiff\varphi_i(z^k)
				+
				\sum\limits_{i\in J}\lambda^k_i\,
					\bigl(\gdiff\varphi_i(z^k)\cup\gdiff(-\varphi_i)(z^k)\bigr),\\
			\label{eq:asymptotic_stationarity_compl_slack}
				&\forall i\in I\colon\quad
				\min(\lambda^k_i,-\varphi_i(\bar z))=0,\\
			\label{eq:asymptotic_stationarity_nonnegativity}
				&\forall i\in J\colon\quad
				\lambda^k_i\geq 0
		\end{align}
	\end{subequations}
	for all $k\in\N$ as well as $z^k\to\bar z$ and $\varepsilon^k\to 0$.
\end{definition}

By definition, each  A$\partial$-stationary point is A$\cdiff$-stationary,
but the converse statement does not hold true in general, see \cref{ex:AM_but_not_AC_regular}.
Referring to the considerations at the beginning of this section, 
we would like to note that due to homogeneity of Clarke's subdifferential,
$\bar z\in\R^n$ is A$\cdiff$-stationary for \eqref{eq:Lipschitzian_program} if and only
if there exist sequences $\{z^k\}_{k\in\N},\{\varepsilon^k\}_{k\in\N}\subset\R^n$ and
$\{\tilde\lambda^k\}_{k\in\N}\subset\R^{p+q}$ which satisfy
$z^k\to\bar z$, $\varepsilon^k\to 0$, as well as
\begin{align*}
	\forall k\in\N\colon\quad
	\varepsilon^k
				\in 
				\partial\varphi_0(z^k)
				+
				\sum\limits_{i\in I\cup J}\tilde\lambda^k_i\,\cdiff\varphi_i(z^k),
	\qquad
				\min(\tilde\lambda^k_i,-\varphi_i(\bar z))=0\quad(i\in I).
\end{align*}
Particularly, the sign condition \eqref{eq:asymptotic_stationarity_nonnegativity} on the
multipliers associated with equality constraints needs to be dropped in this form of the definition.

In the lemma below, we present an equivalent definition of A$\gdiff$-stationarity which
might be more convenient in the light of algorithmic applications since it allows for
certain violations of \eqref{eq:asymptotic_stationarity_compl_slack} and
\eqref{eq:asymptotic_stationarity_nonnegativity}.
\begin{lemma}\label{lem:alternative_definition_of_aymptotic_stationarity}
	A feasible point $\bar z\in\mathcal Z$ of \eqref{eq:Lipschitzian_program} is
	A$\gdiff$-stationary if and only if there are sequences
	$\{z^k\}_{k\in\N},\{\varepsilon^k\}_{k\in\N}\subset\R^n$ and $\{\lambda^k\}_{k\in\N}\subset\R^{p+q}$
	which satisfy $z^k\to\bar z$, $\varepsilon^k\to 0$, 
	\eqref{eq:asymptotic_stationarity_Lagrangian} for each $k\in\N$, as well as
	\begin{subequations}\label{eq:asymptotic_stationarity_alternative}
		\begin{align}
			\label{eq:asymptotic_stationarity_alternative_inequalities}
				&\forall i\in I\colon\quad \lim_{k\to\infty}\min(\lambda^k_i,-\varphi_i(z^k))=0,\\
			\label{eq:asymptotic_stationarity_alternative_equalities}
				&\forall i\in J\colon\quad \liminf_{k\to\infty}\lambda^k_i\geq 0.
		\end{align}
	\end{subequations}
\end{lemma}
\begin{proof}
	$[\Longrightarrow]$: If $\bar z$ is A$\gdiff$-stationary, we find sequences
		$\{z^k\}_{k\in\N},\{\varepsilon^k\}_{k\in\N}\subset\R^n$ and
		$\{\lambda^k\}_{k\in\N}\subset\R^{p+q}$ satisfying \eqref{eq:asymptotic_stationarity}
		for each $k\in\N$, $z^k\to\bar z$, and $\varepsilon^k\to 0$.
		This already yields \eqref{eq:asymptotic_stationarity_alternative_equalities}.
		By continuity of $\varphi_i$ at $\bar z$, 
		we find $\varphi_i(z^k)\to\varphi_i(\bar z)\leq 0$ for each
		$i\in I$. Furthermore, \eqref{eq:asymptotic_stationarity_compl_slack} guarantees
		$\lambda^k_i\geq 0$ for all $k\in\N$ and $i\in I$. For $i\in I\setminus I(\bar z)$, we have
		$\lambda^k_i=0$ and $\varphi_i(z^k)<0$ for all sufficiently large $k\in\N$
		which yields $\min(\lambda^k_i,-\varphi_i(z^k))=0$ for large enough $k\in\N$.
		Fixing $i\in I(\bar z)$, we find $\varphi_i(z^k)\to 0$ which yields
		$\min(\lambda^k_i,-\varphi_i(z^k))\to 0$. This shows validity of
		\eqref{eq:asymptotic_stationarity_alternative_inequalities}.
		
	$[\Longleftarrow]$: Assume that there are sequences 
		$\{z^k\}_{k\in\N},\{\varepsilon^k\}_{k\in\N}\subset\R^n$ and 
		$\{\lambda^k\}_{k\in\N}\subset\R^{p+q}$ which satisfy $z^k\to\bar z$, $\varepsilon^k\to 0$,
		\eqref{eq:asymptotic_stationarity_Lagrangian} for each $k\in\N$, and
		\eqref{eq:asymptotic_stationarity_alternative}.
		Thus, for each $i\in I$, we find a sequence $\{\xi^k_i\}_{k\in\N}\subset\R^n$ with
		$\xi^k_i\in\gdiff\varphi_i(z^k)$ for each $k\in\N$, and for each $i\in J$, we find a
		sequence $\{\eta^k_i\}_{k\in\N}\subset\R^n$ with
		$\eta^k_i\in\gdiff\varphi_i(z^k)\cup\gdiff(-\varphi_i)(z^k)$ for each $k\in\N$ such that
		\[
			\varepsilon^k\in
			\partial\varphi_0(z^k)
			+
			\sum\nolimits_{i\in I}\lambda^k_i\,\xi^k_i
			+
			\sum\nolimits_{i\in J}\lambda^k_i\,\eta^k_i
		\]
		holds for each $k\in\N$. Due to $z^k\to\bar z$ and local Lipschitzness of
		$\varphi_1,\ldots,\varphi_{p+q}$ around $\bar z$, the sequences
		$\{\xi^k_i\}_{k\in\N}$ ($i\in I$) and $\{\eta^k_i\}_{k\in\N}$ ($i\in J$) are bounded.
		
		Assume that there is some index $j\in I$ such that 
		$\min(\lambda^k_j,-\varphi_j(\bar z))<0$ holds along a subsequence (without relabeling).
		Due to $\varphi_j(\bar z)\leq 0$, this yields $\lambda^k_j<0$ for all $k\in\N$.
		Invoking \eqref{eq:asymptotic_stationarity_alternative_inequalities} and 
		$\varphi_j(z^k)\to\varphi_j(\bar z)$, we find $\lambda^k_j\to 0$. 
		This yields $\lambda^k_j\,\xi^k_j\to 0$, and defining 
		$\tilde\varepsilon^k:=\varepsilon^k-\lambda^k_j\,\xi^k_j$ 
		as well as $\tilde\lambda^k_j:=0$ for each $k\in\N$ yields $\tilde\varepsilon^k\to 0$
		and
		\begin{equation}\label{eq:adjusted_explicit_asymptotic_condition}
			\tilde\varepsilon^k
			\in 
			\partial\varphi_0(z^k)
			+
			\sum\nolimits_{i\in I\setminus\{j\}}\lambda^k_i\,\xi^k_i
			+
			\sum\nolimits_{i\in J}\lambda^k_i\,\eta^k_i
			+
			\tilde\lambda^k_j\,\gdiff\varphi_j(z^k)
		\end{equation}
		as well as $\min(\tilde\lambda^k_j,-\varphi_j(\bar z))=0$ for each $k\in\N$.
		
		Next, we assume that there is some $j\in I$ such that
		$\min(\lambda^k_j,-\varphi_j(\bar z))>0$ holds along a subsequence (without relabeling).
		This yields $\lambda^k_j>0$ for all $k\in\N$ and $j\in I\setminus I(\bar z)$. Due to
		$\varphi_j(z^k)\to\varphi_j(\bar z)<0$,
		\eqref{eq:asymptotic_stationarity_alternative_inequalities}
		yields $\lambda^k_j\to 0$, i.e., $\lambda^k_j\,\xi^k_j\to 0$. Again, we set
		$\tilde\varepsilon^k:=\varepsilon^k-\lambda^k_j\,\xi^k_j$ and $\tilde\lambda^k_j:=0$ in order
		to find \eqref{eq:adjusted_explicit_asymptotic_condition} and
		$\min(\tilde\lambda^k_j,-\varphi_j(\bar z))=0$ for each $k\in\N$ as well as
		$\tilde\varepsilon^k\to 0$.
		
		Finally, assume that there is some $j\in J$ such that $\lambda^k_j<0$ holds along
		a subsequence (without relabeling). 
		Then \eqref{eq:asymptotic_stationarity_alternative_equalities} guarantees
		$\lambda^k_j\to 0$ which yields $\lambda^k_j\,\eta^k_j\to 0$. We set
		$\tilde\varepsilon^k:=\varepsilon^k-\lambda^k_j\,\eta^k_j$ and $\tilde\lambda^k_j:=0$
		and find
		\[
			\tilde\varepsilon^k
			\in 
			\partial\varphi_0(z^k)
			+
			\sum\nolimits_{i\in I}\lambda^k_i\,\xi^k_i
			+
			\sum\nolimits_{i\in J\setminus\{j\}}\lambda^k_i\,\eta^k_i
			+
			\tilde\lambda^k_j\,(\gdiff\varphi_j(z^k)\cup\gdiff(-\varphi_j)(z^k))
		\]
		for each $k\in\N$ as well as $\tilde\varepsilon^k\to 0$.
		
		Performing the above transformations iteratively for each index $j\in I\cup J$
		where a violation of \eqref{eq:asymptotic_stationarity_compl_slack}
		or \eqref{eq:asymptotic_stationarity_nonnegativity} occurs, we
		can \emph{hide} these asymptotic violations, restricted via
		\eqref{eq:asymptotic_stationarity_alternative_inequalities} and
		\eqref{eq:asymptotic_stationarity_alternative_equalities}, in the definition
		of $\{\varepsilon^k\}_{k\in\N}$.
		Hence, $\bar z$ is A$\gdiff$-stationary.
\end{proof}

Let us now invoke \cref{prop:sequential_stationarity_geometric_constraints}.
Exploiting the sum rule and positive homogeneity of the limiting subdifferential,
the following result follows easily by similar considerations as presented after
\cref{def:C_M_stationarity}.
\begin{theorem}\label{thm:local_minimizers_are_asymptotically_stationary}
	If $\bar z\in\mathcal Z$ is a local minimizer of \eqref{eq:Lipschitzian_program}, 
	then it is an A$\partial$-stationary point of this program.
\end{theorem}

From the above result, we immediately see that each local minimizer of \eqref{eq:Lipschitzian_program}
is A$\cdiff$-stationary as well.
Using Clarke's subdifferential, approximate KKT-type necessary optimality conditions
can be found in \cite{DuttaDebTulshyanArora2013}. 
Furthermore, we would like to mention the recently published paper \cite{HelouSantosSimoes2020} 
where Goldstein's $\varepsilon$-subdifferential construction
is used to design a sequential stationarity condition for Lipschitzian programs. 
The authors stated an implementable algorithm which computes asymptotically stationary points
in their sense. On the other hand, Goldstein's $\varepsilon$-subdifferential is even larger than
Clarke's subdifferential (w.r.t.\ set inclusion) and, thus, provides very weak stationarity 
conditions. Furthermore, its numerical computation is quite challenging since it is likely to
be set-valued for each point from the underlying function's domain.

Next, we present a simple observation regarding the sequential stationarity notions from
\cref{def:asymptotic_stationarity}. Its proof is based on the outer semicontinuity of
the Clarke and limiting subdifferential as well as their local boundedness for locally
Lipschitzian functions and, for large parts, can be distilled from
\cite[Lemma~3.4]{Mehlitz2020b}. It basically says that each asymptotically $\gdiff$-stationary
point of \eqref{eq:Lipschitzian_program} satisfies a Fritz--John-type condition based on the
subdifferential construction $\gdiff$. Simple examples indicate, however, that A$\gdiff$-stationarity
is, in general, stronger than these Fritz--John-type conditions.
\begin{lemma}\label{lem:Fritz_John_conditions}
	Let $\bar z\in\mathcal Z$ be an A$\gdiff$-stationary point of
	\eqref{eq:Lipschitzian_program} such that the sequences 
	$\{z^k\}_{k\in\N},\{\varepsilon^k\}_{k\in\N}\subset\R^n$ and $\{\lambda^k\}_{k\in\N}\subset\R^{p+q}$
	with $z^k\to\bar z$ and $\varepsilon^k\to 0$
	satisfy \eqref{eq:asymptotic_stationarity} for each $k\in\N$.
	Then the following assertions hold.
	\begin{enumerate}
		\item If $\{\lambda^k\}_{k\in\N}$ is bounded, then $\bar z$ is $\gdiff$-stationary.
		\item If $\{\lambda^k\}_{k\in\N}$ is not bounded, then we find a nonzero vector
			$\lambda\in\R^{p+q}$ which
			satisfies \eqref{eq:C_M_stationarity_compl_slack},
			\eqref{eq:C_M_stationarity_equality_constraints}, and
			\[
				0\in
				\sum\limits_{i\in I}\lambda_i\,\gdiff\varphi_i(\bar z)
				+
				\sum\limits_{i\in J}\lambda_i\,
					\bigl(\gdiff\varphi_i(\bar z)\cup\gdiff(-\varphi_i)(\bar z)\bigr).
			\]
	\end{enumerate}
\end{lemma}

Below, we are going to interrelate $\gdiff$- and A$\gdiff$-stationary
points of \eqref{eq:Lipschitzian_program}. 
Therefore, we fix a feasible point $\bar z\in\mathcal Z$ of \eqref{eq:Lipschitzian_program}.
Let us introduce a set-valued mapping $\MM^\square\colon\R^n\tto\R^n$ by means of
\[
	\MM^{\square}(z):=
	\left\{
		\sum_{i\in I(\bar z)}\lambda_i\,\gdiff\varphi_i(z)
		+
		\sum_{i\in J}\lambda_i\,\bigl(\gdiff\varphi_i(z)\cup\gdiff(-\varphi_i)(z)\bigr)
		\,\middle|\,
			\lambda_i\geq 0\,(i\in I(\bar z)\cup J)
	\right\}
\]
for each $z\in\R^n$. Note that this map explicitly depends on $\bar z$ since the set
$I(\bar z)$ appears.
The definition of $\MM^{\square}$ directly shows that $\bar z$ is $\gdiff$-stationary
for \eqref{eq:Lipschitzian_program} if and only if 
$\partial\varphi_0(\bar z)\cap(-\MM^{\square}(\bar z))\neq\varnothing$ holds.
Furthermore, by definition of A$\gdiff$-stationarity, we find the following result.
Its proof is analogous to the one of \cite[Lemma~3.6]{Mehlitz2020b} and, again, basically exploits
outer semicontinuity and local boundedness of the respective subdifferential as well as the
definition of the outer set limit.
\begin{lemma}\label{lem:AC_stationarity_via_M}
	Let $\bar z\in\mathcal Z$ be a feasible point of \eqref{eq:Lipschitzian_program}.
	Then the following assertions hold.
	\begin{enumerate}
		\item[(a)] If $\bar z$ is an A$\gdiff$-stationary point 
			of \eqref{eq:Lipschitzian_program}, then 
			$\partial\varphi_0(\bar z)\cap\bigl(-\limsup\nolimits_{z\to\bar z}\MM^{\square}(z)\bigr)
				\neq\varnothing$.
		\item[(b)] If $\varphi_0$ is continuously differentiable at $\bar z\in\R^n$ while
			$-\nabla \varphi_0(\bar z)\in\limsup\nolimits_{z\to\bar z}\MM^{\square}(z)$
			is valid, then $\bar z$ is an A$\gdiff$-stationary 
			point of \eqref{eq:Lipschitzian_program}.
	\end{enumerate}
\end{lemma}

Taking \cref{thm:local_minimizers_are_asymptotically_stationary} and
\cref{lem:AC_stationarity_via_M} together, the definition of the following constraint
qualifications is reasonable.
\begin{definition}\label{def:asymptotic_regularity}
	Let $\bar z\in\mathcal Z$ be a feasible point of \eqref{eq:Lipschitzian_program}.
	\begin{enumerate}
		\item We call $\bar z$ \emph{asymptotically $\gdiff$-regular} (A$\gdiff$-regular for short)
			whenever the condition
			$\limsup\nolimits_{z\to\bar z}\MM^{\square}(z)\subset \MM^\square(\bar z)$
			is valid, i.e., if $\MM^\square$ is outer semicontinuous at $\bar z$.
		\item We call $\bar z$ \emph{weakly asymptotically $\partial$-regular}
			(wA$\partial$-regular for short)
			whenever the condition
			$\limsup\nolimits_{z\to\bar z}\MM(z)\subset \MM^\textup{c}(\bar z)$
			is valid.
	\end{enumerate}
\end{definition}

We would like to point out that A$\gdiff$-regularity and wA$\partial$-regularity reduce 
to the so-called \emph{cone continuity property} from \cite[Definition~3.1]{AndreaniMartinezRamosSilva2016},
which is sometimes referred to as AKKT-regularity, whenever the functions 
$\varphi_i$ ($i\in I(\bar z)\cup J$) are continuously differentiable at $\bar z$.
In the general nonsmooth setting, however, we only get the relations
\[
	\text{A$\partial$-regularity}\quad\Longrightarrow\quad\text{wA$\partial$-regularity,}
	\qquad
	\text{A$\cdiff$-regularity}\quad\Longrightarrow\quad\text{wA$\partial$-regularity.}
\]
The following examples underline that A$\partial$-regularity and A$\cdiff$-regularity
are independent of each other. Note that both of these examples are stated in the context of
complementarity-constrained optimization, see \cref{sec:MPCCs} as well.
Actually, \cref{ex:AM_but_not_AC_regular} is taken from 
\cite[Example~6]{AndreaniHaeserSecchinSilva2019} where it is used to visualize 
closely related issues.
Finally, let us mention that these examples also indicate that wA$\partial$-regularity is strictly
weaker than $\partial$- and $\cdiff$-regularity.
\begin{example}\label{ex:AC_but_not_AM_regular}
	We consider the feasible region $\mathcal Z\subset\R^2$ modeled by
	\[
		\mathcal Z:=\{z\in\R^2\,|\,\varphi_1(z):=z_1^3-z_2\leq 0,\,\varphi_2(z):=\min(z_1,z_2)=0\}
	\]
	at $\bar z:=(0,0)$. 
	For the computation of the subdifferentials associated with $\varphi_2$, we refer the reader
	to \cref{lem:composed_subdifferential_of_minimum_function}. 
	Exploiting $z^k:=((3k)^{-1/2},0)$, we find
	\[
		(1,1)
		=
		k\,(1/k,-1)+(k+1)\,(0,1)
		\in
		k\,\nabla\varphi_1(z^k)+(k+1)\partial\varphi_2(z^k)
		\subset
		\MM(z^k)
	\]
	for each $k\in\N$. Thus, due to
	$\MM(\bar z)=\{\eta\in\R^2\,|\,\eta_1=0\,\lor\,\eta_2\leq 0\}$,
	$\bar z$ is not A$\partial$-regular.
	
	On the other hand, we find
	$\MM^\textup{c}(\bar z)=\{\eta\in\R^2\,|\,\eta_1\geq 0\,\lor\,\eta_2\leq 0\}$.
	Suppose now that there is some $\eta\in\limsup_{z\to\bar z}\MM^\textup{c}(z^k)$ which satisfies
	$\eta_1<0$ and $\eta_2>0$, i.e., that $\bar z$ is not 
	A$\cdiff$-regular. Then we find $\{z^k\}_{k\in\N},\{\eta^k\}_{k\in\N}\subset\R^2$ 
	such that $z^k\to\bar z$, $\eta^k\to\eta$, and
	$\eta^k\in\MM^\textup{c}(z^k)$ for all $k\in\N$, i.e., there are sequences 
	$\{\lambda^k_1\}_{k\in\N},\{\lambda^k_2\}_{k\in\N}\subset\R_+$ and
	$\{\xi^k\}_{k\in\N}\subset\R^2$ such that
	\[
		\forall k\in\N\colon\quad
		\eta^k=\lambda^k_1(3(z^k_1)^2,-1)
		+
		\lambda^k_2\,\xi^k,\qquad
		\xi^k\in\cdiff\varphi_2(z^k)\cup\cdiff(-\varphi_2)(z^k).
	\]
	Due to $\eta^k_1<0$ for sufficiently large $k\in\N$, we find $\xi^k_1<0$ for sufficiently large
	$k\in\N$. This is only
	possible if $z^k_1\leq z^k_2$ is valid for sufficiently large $k\in\N$.
	In these situations, we have $\xi^k_2\leq 0$ as well, i.e., $\eta^k_2\leq 0$ follows for
	large enough $k\in\N$. This, however, contradicts $\eta_2>0$.
	Thus, $\bar z$ is A$\cdiff$-regular.
\end{example}

\begin{example}\label{ex:AM_but_not_AC_regular}
	Let us investigate the feasible region $\mathcal Z\subset\R^3$ given by
	\[
		\mathcal Z:=
		\left\{z\in\R^3\,\middle|\,
			\begin{aligned}
				&\varphi_1(z):=-z_1\leq 0,\,\varphi_2(z):=-z_3\leq 0,\\
				&\varphi_3(z):=\min(z_1^3+z_2+z_3,z_1^3-z_2+z_3)=0
			\end{aligned}
		\right\}
	\]
	at $\bar z:=(0,0,0)$. For the computation of the subdifferentials associated with
	$\varphi_3$, we make use of \cref{lem:composed_subdifferential_of_minimum_function}
	again.
	Considering $z^k:=((3k/2)^{-1/2},0,0)$, we find
	\begin{align*}
		(1,0,0)
		&=
		(-1,0,0)+k\,(0,0,-1)+k\,(2/k,0,1)\\
		&\in 
		\nabla\varphi_1(z^k)+k\,\nabla\varphi_2(z^k)+k\,\cdiff\varphi_3(z^k)
		\subset
		\MM^\textup{c}(z^k)
	\end{align*}
	for each $k\in\N$. On the other hand, a simple calculation reveals 
	$\MM^\textup{c}(\bar z)\subset\R_-\times\R^2$ which means that $\bar z$ cannot be 
	A$\cdiff$-regular.
	
	One can show that $\MM(\bar z)=\{\eta\in\R^3\,|\,\eta_1\leq 0,\,\eta_3\leq|\eta_2|\}$ holds.
	Fix some $\eta\in\limsup_{z\to\bar z}\MM(z)$. 
	Then we find sequences $\{z^k\}_{k\in\N},\{\eta^k\}_{k\in\N}\subset\R^3$ such that
	$z^k\to\bar z$, $\eta^k\to\eta$, and $\eta^k\in\MM(z^k)$ for each $k\in\N$.
	One can easily check that
	\[
		\partial\varphi_3(z)\cup\partial(-\varphi_3)(z)
		\subset
		\{(3z_1^2,\pm1,1)\}\cup\{(-3z_1^2,\alpha,-1)\,|\,\alpha\in[-1,1]\}
	\]
	holds for all $z\in\R^3$, and this reveals
	that $\eta^k_3\leq|\eta^k_2|$ is valid for all $k\in\N$.
	Taking the limit $k\to\infty$ yields $\eta_3\leq|\eta_2|$.
	Supposing that $\eta_1>0$ holds, there are sequences
	$\{\lambda^k_1\}_{k\in\N},\{\lambda^k_2\}_{k\in\N},\{\lambda^k_3\}_{k\in\N}\subset\R_+$ such that
	\[
		\eta^k=\lambda^k_1(-1,0,0)+\lambda^k_2(0,0,-1)+\lambda^k_3(3(z_1^k)^2,\pm1,1)
	\]
	is valid for sufficiently large $k\in\N$. Inspecting the second component, $\{\lambda^k_3\}_{k\in\N}$
	needs to be convergent since $\{\eta^k_2\}_{k\in\N}$ converges to $\eta_2$.
	Thus, we find $\lambda^k_3\,3(z_1^k)^2\to 0$ by $z^k_1\to 0$, and due
	to $\eta^k_1\to\eta_1$, this leads to $\eta_1\leq 0$ - a contradiction.
	As a consequence, $\bar z$ is A$\partial$-regular.
	
	Note that using the objective function given by $\varphi_0(z):=-z_1$ for all $z\in\R^3$,
	one can exploit \cref{lem:AC_stationarity_via_M} in order to see that $\bar z$ is an
	A$\cdiff$-stationary point of the associated program \eqref{eq:Lipschitzian_program}
	which is not A$\partial$-stationary.
\end{example}

Exploiting \cref{lem:AC_stationarity_via_M}, we find that each A$\gdiff$-stationary
point of \eqref{eq:Lipschitzian_program}, which is A$\gdiff$-regular, is already 
$\gdiff$-stationary. More precisely, A$\gdiff$-regularity is the weakest condition which implies
that an A$\gdiff$-stationary point is actually $\gdiff$-stationary.
In the light of \cite[Section~1]{AndreaniMartinezRamosSilva2016}, we may thus 
refer to A$\gdiff$-regularity
as a \emph{strict} constraint qualification.
Taking \cref{thm:local_minimizers_are_asymptotically_stationary} and \cref{lem:AC_stationarity_via_M}
together, we find the following result.
\begin{theorem}\label{thm:AC_regular_local_minimizers}
	Let $\bar z\in\mathcal Z$ be a local minimizer of
	\eqref{eq:Lipschitzian_program}.
	Then the following assertions hold.
	\begin{enumerate}
		\item[(a)] If $\bar z$ is wA$\partial$-regular, then $\bar z$ is $\cdiff$-stationary.
		\item[(b)] If $\bar z$ is A$\partial$-regular, then $\bar z$ is $\partial$-stationary.
	\end{enumerate}
\end{theorem}

Exploiting the sequential stationarity condition based on the regular subdifferential mentioned at
the end of \cref{sec:geometric_constraints}, it is technically possible to introduce asymptotic
regularity conditions based on the regular subdifferential as well.
However, let us note that this approach comes along with two disadvantages.
First, the regular subdifferential does not obey a classical but only a fuzzy sum rule which is why 
the outer semicontinuity properties of much more difficult set-valued mappings would need to be considered.
Second, let us recall that in contrast to the limiting and Clarke subdifferential, the regular subdifferential
is not outer semicontinuous in general, see \cref{sec:variational_analysis}, so that a constraint
qualification for stationarity in terms of the regular subdifferential is likely to fail anyway.

In the remaining part of this section, we are going to embed the constraint qualifications from
\cref{def:asymptotic_regularity} into the landscape of qualification conditions from nonsmooth
optimization.
Let us recall that $\gdiff$-NMFCQ, the 
\emph{nonsmooth Mangasarian--Fromovitz constraint qualification w.r.t.\ $\gdiff$}
is said to hold at $\bar z\in\mathcal Z$
whenever the condition
\[
	\left.
	\begin{aligned}
		&0\in\sum\limits_{i\in I(\bar z)}\lambda_i\,\gdiff\varphi_i(\bar z)
		+\sum\limits_{i\in J}
			\lambda_i\,\bigl(\gdiff\varphi_i(\bar z)\cup\gdiff(-\varphi_i)(\bar z)\bigr),\\
		&\lambda_i\geq 0\,(i\in I(\bar z)\cup J)
	\end{aligned}
	\right\}
	\ \Longrightarrow \
		\lambda_i=0\,(i\in I(\bar z)\cup J)
\]
is valid. Clearly, this reduces to the classical MFCQ when continuously differentiable functions
$\varphi_1,\ldots,\varphi_{p+q}$ are under consideration.
Exploiting local boundedness and outer semicontinuity of
the Clarke and limiting subdifferential, standard arguments show that $\gdiff$-NMFCQ is sufficient for 
A$\gdiff$-regularity. 
However, the study from \cite[Section~4]{AndreaniMartinezRamosSilva2016} for smooth functions
clearly underlines that A$\gdiff$-regularity should be much weaker than $\gdiff$-NMFCQ in general. 
Below, we visualize this with the aid of two examples.

First, we want to review the nonsmooth variant of the \emph{relaxed constant positive linear
dependence constraint qualification} introduced in
\cite[Definition~1.1]{XuYe2020} via limiting normals.
\begin{definition}\label{def:RCPLD}
	Let $\bar z\in\mathcal Z$ be a feasible point of \eqref{eq:Lipschitzian_program} and assume that
	the functions $\varphi_{p+1},\ldots,\varphi_{p+q}$, that correspond to the equality constraints
	in \eqref{eq:Lipschitzian_program}, are continuously differentiable in
	a neighborhood of $\bar z$. 
	We say that $\gdiff$-RCPLD, the \emph{relaxed constant positive linear dependence
	constraint qualification w.r.t.\ $\gdiff$}, holds at $\bar z$ whenever the following conditions
	are valid.
	\begin{enumerate}
		\item[(i)] The family $(\nabla \varphi_i(z))_{i\in J}$ has constant rank in some
			neighborhood of $\bar z$.
		\item[(ii)]  There is some index set $\widetilde J\subset J$ such that
			$(\nabla\varphi_i(\bar z))_{i\in \widetilde J}$ is a basis of the subspace
			$\operatorname{span}\{\nabla\varphi_i(\bar z)\,|\,i\in J\}$.
		\item[(iii)] For each index set $\widetilde I\subset I(\bar z)$ and each family of subgradients
			$(\xi_i)_{i\in \widetilde I}$ satisfying $\xi_i\in\gdiff\varphi_i(\bar z)$
			($i\in \widetilde I$) such that the pair of families 
			$((\xi_i)_{i\in \widetilde I},(\nabla\varphi_i(\bar z))_{i\in\widetilde J})$
			is positive linearly dependent,
			we can ensure that, for large enough $k\in\N$, the vectors from the family
			$(\xi_i^k)_{i\in \widetilde I}\cup(\nabla\varphi_i(z^k))_{i\in\widetilde J}$ are
			linearly dependent where the sequences $\{z^k\}_{k\in\N}\subset\R^n$ and
			$\{\xi_i^k\}_{k\in\N}\subset\R^n$ ($i\in \widetilde I$) with $z^k\to\bar z$,
			$\xi_i^k\to\xi_i$ ($i\in \widetilde I$) and $\xi_i^k\in\gdiff\varphi_i(z^k)$ for all 
			$k\in\N$ and $i\in \widetilde I$ are arbitrarily chosen.
	\end{enumerate}
\end{definition}

By definition, we see that $\gdiff$-RCPLD is milder than $\gdiff$-NMFCQ whenever the
equality constraints under consideration are smooth.
Furthermore, $\cdiff$-RCPLD is sufficient for $\partial$-RCPLD.
The converse is not true which can be exemplary seen when considering the constraint
region modeled by the single inequality constraint $-|z|\leq 0$
at $\bar z:=0$. Due to $\partial(-\abs{\cdot})(0)=\{-1,1\}$ and $\cdiff(-\abs{\cdot})(0)=[-1,1]$,
$\partial$-NMFCQ is valid which guarantees validity of $\partial$-RCPLD.
On the other hand, $\cdiff$-RCPLD obviously fails to hold since $0\in\cdiff(-\abs{\cdot})(0)$ while
$\{1/k\}_{k\in\N}\subset\cdiff(-\abs{\cdot})(0)$ is non-vanishing and satisfies $1/k\to 0$.

Let us mention that RCPLD has been introduced for standard nonlinear programs
in \cite{AndreaniHaeserSchuverdtSilva2012}, and assuming that all the functions
$\varphi_i$ ($i\in I\cup J$) are smooth, \cref{def:RCPLD} recovers this classical notion.
Recently, the definition of RCPLD has been extended to nonsmooth and complementarity-based
systems in \cite{ChieuLee2013,GuoLin2013,XuYe2020}. In \cite{MehlitzMinchenko2021},
a parametric version of this constraint qualification has been discussed.

Below, we adapt the proofs of \cite[Theorem~4.8]{AndreaniHaeserSecchinSilva2019} and
\cite[Theorem~4.2]{Ramos2019} in order to verify that
$\gdiff$-RCPLD is indeed sufficient for A$\gdiff$-regularity.
\begin{lemma}\label{lem:RCPLD}
	Let $\bar z\in\mathcal Z$ be a feasible point of \eqref{eq:Lipschitzian_program} and assume that
	the functions $\varphi_{p+1},\ldots,\varphi_{p+q}$ are continuously differentiable in
	a neighborhood of $\bar z$. 
	Furthermore, let $\gdiff$-RCPLD hold at $\bar z$.
	Then $\bar z$ is A$\gdiff$-regular.
\end{lemma}
\begin{proof}
	Let us fix some arbitrary point $\eta\in\limsup_{z\to\bar z}\MM^{\square}(z)$.
	Then we find sequences $\{z^k\}_{k\in\N},\{\eta^k\}_{k\in\N}\subset\R^n$ such that
	$z^k\to\bar z$, $\eta^k\to \eta$, and $\eta^k\in\MM^\square(z^k)$ for all $k\in\N$.
	By construction, there are sequences $\{\lambda_i^k\}_{k\in\N}\subset\R_+$ ($i\in I(\bar z)$), 
	$\{\xi_i^k\}_{k\in\N}\subset\R^n$ ($i\in I(\bar z)$), and $\{\mu^k_i\}_{k\in\N}\subset\R$ ($i\in J$)
	which satisfy $\xi_i^k\in\gdiff\varphi_i(z^k)$ for all $i\in I(\bar z)$ and $k\in\N$ as well as
	\[
		\eta^k
		=
		\sum\nolimits_{i\in I(\bar z)}\lambda_i^k\,\xi_i^k
		+
		\sum\nolimits_{i\in J}\mu_i^k\,\nabla\varphi_i(z^k)
	\]
	for all $k\in\N$. Exploiting validity of $\gdiff$-RCPLD (we use the notation from
	\cref{def:RCPLD}), we find sequences $\{\tilde\mu_i^k\}_{k\in\N}\subset\R$ ($i\in\widetilde J$) such
	that
	\[
		\eta^k
		=
		\sum\nolimits_{i\in I(\bar z)}\lambda_i^k\,\xi_i^k
		+
		\sum\nolimits_{i\in \widetilde J}\tilde\mu_i^k\nabla\varphi_i(z^k)
	\]
	is valid for all $k\in\N$. 
	Next, for each $k\in\N$, we apply \cite[Lemma~1]{AndreaniHaeserSchuverdtSilva2012} in order to find
	an index set $I^k\subset I(\bar z)$ as well as multipliers $\hat\lambda^k_i>0$ ($i\in I^k$)
	and $\hat\mu^k_i\in\R$ ($i\in\widetilde J$) such that
	\begin{equation}\label{eq:suitable_representation_RCPLD_proof}
		\eta^k
		=
		\sum\nolimits_{i\in I^k}\hat\lambda_i^k\,\xi_i^k
		+
		\sum\nolimits_{i\in \widetilde J}\hat\mu_i^k\,\nabla\varphi_i(z^k)
	\end{equation}
	while the vectors from $(\xi_i^k)_{i\in I^k}\cup(\nabla\varphi_i(z^k))_{i\in\widetilde J}$
	are linearly independent.
	Noting that there are only finitely many subsets of $I(\bar z)$, we may assume w.l.o.g.\ that
	$I^k=\widetilde I$ holds for all $k\in\N$ and some $\widetilde I\subset I(\bar z)$ 
	(if necessary, consider a suitable 	subsequence).
	
	Supposing that $\{((\hat\lambda_i^k)_{i\in \widetilde I},(\hat\mu_i^k)_{i\in\widetilde J})\}_{k\in\N}$ 
	is not bounded,
	we can divide \eqref{eq:suitable_representation_RCPLD_proof} by its norm along a
	principally divergent subsequence (without relabeling).
	Afterwards, we take the limit $k\to\infty$. Observing that
	$\{\xi_i^k\}_{k\in\N}$ ($i\in \widetilde I$) 
	is bounded by local Lipschitzness of $\varphi_i$ at $\bar z$,
	while $\varphi_i$ ($i\in\widetilde J$) is continuously differentiable at $\bar z$, we can
	exploit the outer semicontinuity of the subdifferential construction $\gdiff$ in order to find
	$\hat\lambda_i\geq 0$ and $\xi_i\in\gdiff\varphi_i(\bar z)$ such that, along a 
	suitably chosen subsequence, $\xi^k_i\to\xi_i$ ($i\in \widetilde I$) as well as 
	$\hat\mu_i\in\R$ ($i\in\widetilde J$) which satisfy
	\[
		0
		=
		\sum\nolimits_{i\in \widetilde I}\hat\lambda_i\,\xi_i
		+
		\sum\nolimits_{i\in \widetilde J}\hat\mu_i\,\nabla\varphi_i(\bar z)
	\]
	while $\hat\lambda_i$ ($i\in\widetilde{I}$) and $\hat\mu_i$ ($i\in\widetilde J$) are not all
	vanishing.
	Thus, the pair of families 
	$((\xi_i)_{i\in \widetilde I},(\nabla\varphi_i(\bar z))_{i\in\widetilde J})$
	is positive linearly dependent. By definition of $\gdiff$-RCPLD, this is a contradiction to
	the linear independence of 
	$(\xi_i^k)_{i\in \widetilde I}\cup(\nabla\varphi_i(z^k))_{i\in\widetilde J}$
	for each $k\in\N$ which has been shown earlier.
	
	Thus, the sequence 
	$\{((\hat\lambda_i^k)_{i\in \widetilde I},(\hat\mu_i^k)_{i\in\widetilde J})\}_{k\in\N}$
	is bounded. Similar arguments as above now yields the existence of $\lambda_i\geq 0$ and
	$\xi_i\in\gdiff\varphi_i(\bar z)$ ($i\in \widetilde I$) as well as $\mu_i\in\R$ ($i\in\widetilde J$)
	such that
	\[
		\eta
		=
		\sum\nolimits_{i\in \widetilde I}\lambda_i\,\xi_i
		+
		\sum\nolimits_{i\in\widetilde J}\mu_i\,\nabla\varphi_i(\bar z)
	\]
	follows by taking the limit $k\to\infty$ in \eqref{eq:suitable_representation_RCPLD_proof}
	along a suitable subsequence. Thus, we found $\eta\in\MM^\square(\bar z)$, and this shows
	validity of A$\gdiff$-regularity.
\end{proof}

We would like to mention the following result which is related to
\cite[Theorem~3.10]{Mehlitz2020b}. It addresses the situation where the 
functions $\varphi_1,\ldots,\varphi_{p+q}$ are piecewise affine which is
a classical regular situation in nonsmooth analysis,
see \cite{Robinson1981}.
\begin{lemma}\label{lem:affine_constraints_AC_regular}
	Let $\bar z\in\mathcal Z$ be a feasible point of \eqref{eq:Lipschitzian_program} such that
	all the functions $\varphi_i$ ($i\in I(\bar z)\cup J$) are piecewise affine in a
	neighborhood of $\bar z$. 
	Then $\bar z$ is A$\gdiff$-regular.
\end{lemma}
\begin{proof}
	Let us first show that $\bar z$ is A$\partial$-regular.
	Fix some $i\in I(\bar z)\cup J$.
	Since $\varphi_i$ is piecewise affine in a neighborhood of $\bar z$, there only
	exist finitely many different regular (and, thus, limiting) 
	normal cones to $\epi\varphi_i$ locally around
	$\bar z$. 
	Thus, we find a neighborhood $U$ of $\bar z$ and finitely many closed sets 
	$C^1,\ldots,C^\ell\subset\R^n$
	such that, for each $z\in U$, there exists $\nu\in\{1,\ldots,\ell\}$ satisfying
	$\partial\varphi_i(z)=C^\nu$. Let us show $\partial\varphi_i(z)\subset\partial\varphi_i(\bar z)$
	for all $z\in V$ where $V\subset U$ is some neighborhood of $\bar z$. 
	Supposing that this is not true, we find a sequence $\{z^k\}_{k\in\N}\subset\R^n$ and
	some $\nu\in\{1,\ldots,\ell\}$ such that $z^k\to\bar z$, $\partial\varphi_i(z^k)=C^\nu$ for all
	$k\in\N$, and $C^\nu\setminus\partial\varphi_i(\bar z)\neq\varnothing$ hold. 
	On the other hand, by outer
	semicontinuity of the limiting subdifferential and $z^k\to\bar z$, 
	we automatically have $C^\nu\subset\partial\varphi_i(\bar z)$, which is a contradiction.
	
	Observing that there are only finitely many indices in $I(\bar z)\cup J$, we, thus, find
	a neighborhood $\mathcal U$ of $\bar z$ such that 
	$\partial\varphi_i(z)\subset\partial\varphi_i(\bar z)$
	holds for all $z\in\mathcal U$ and all $i\in I(\bar z)\cup J$ while
	$\partial(-\varphi_i)(z)\subset\partial(-\varphi_i)(\bar z)$ is valid for all
	$z\in\mathcal U$ and all $i\in J$.
	Thus, we automatically have $\MM(z)\subset\MM(\bar z)$ for all $z\in\mathcal U$,
	and this yields A$\partial$-regularity of $\bar z$.
	
	Recalling that Clarke's subdifferential corresponds to the closed convex hull of the limiting
	subdifferential, we infer $\MM^\textup{c}(z)\subset\MM^{\textup{c}}(\bar z)$ for all $z\in\mathcal U$
	from above, and this yields A$\cdiff$-regularity of $\bar z$.
\end{proof}

Observe that in contrast to $\gdiff$-RCPLD, we do not need to assume any smoothness of
the functions $\varphi_{p+1},\ldots,\varphi_{p+q}$, which model the equality constraints in 
\eqref{eq:Lipschitzian_program}, in \cref{lem:affine_constraints_AC_regular}.
Furthermore, we would like to note that even in the situation where $\varphi_1,\ldots,\varphi_p$
are piecewise affine while $\varphi_{p+1},\ldots,\varphi_{p+q}$ are affine, $\gdiff$-RCPLD
does not necessarily hold. 
In order to see this, one could simply consider the feasible set $\mathcal Z$ modeled by the single
inequality constraint $\abs{z}\leq 0$ at $\bar z:=0$. 
This is essentially different from the observations in standard nonlinear optimization
where fully affine systems satisfy the \emph{constant rank constraint qualification} and, thus,
RCPLD, see \cite[Section~3]{AndreaniHaeserSchuverdtSilva2012}.

The above observations underline that A$\gdiff$-regularity is a very weak 
constraint qualification
which may hold even for highly degenerated programs where standard constraint qualifications
like $\gdiff$-NMFCQ are always violated.
In \cref{fig:CQs}, we summarize the relations between the discussed constraint qualifications. 

\begin{figure}[h]
\centering
\begin{tikzpicture}[->]

  \node[punkt] at (1,1) 	(A){piecewise\\ affine data};
  \node[punkt] at (4,1) 	(B){$\cdiff$-NMFCQ};
  \node[punkt] at (8,0) 	(C){$\partial$-NMFCQ};
  \node[punkt] at (11,1)	(D){piecewise\\ affine data};
  \node[punkt] at (4.5,-1)	(E){$\cdiff$-RCPLD};
  \node[punkt] at (7.5,-2)	(F){$\partial$-RCPLD};
  \node[punkt] at (4,-4)	(G){A$\cdiff$-regularity};
  \node[punkt] at (8,-4)    (H){A$\partial$-regularity};
  \node[punkt] at (6,-6) 	(I){wA$\partial$-regularity};

  \path     (A) edge[-implies,thick,double] node {}(G)
            (B) edge[-implies,thick,double] node {}(C)
            (B) edge[-implies,thick,double] node[right] {$(*)$}(E)
            (B) edge[-implies,thick,double,bend right] node {}(G)
            (C) edge[-implies,thick,double] node[left] {$(*)$}(F)
            (C) edge[-implies,thick,double,bend left] node {}(H)
            (D) edge[-implies,thick,double] node {}(H)
            (E) edge[-implies,thick,double] node [above] {$(*)$}(F)
            (E) edge[-implies,thick,double] node [right] {$(*)$}(G)
            (F) edge[-implies,thick,double] node [left] {$(*)$}(H)
            (G) edge[-implies,thick,double] node {}(I)
            (H) edge[-implies,thick,double] node {}(I)
            ;
\end{tikzpicture}
\caption{
	Relations between constraint qualifications addressing \eqref{eq:Lipschitzian_program}.
	The label $(*)$ indicates that the underlying relation is only reasonable whenever
	the functions $\varphi_{p+1},\ldots,\varphi_{p+q} $ are continuously differentiable at
	the point of interest.
}
\label{fig:CQs}
\end{figure}
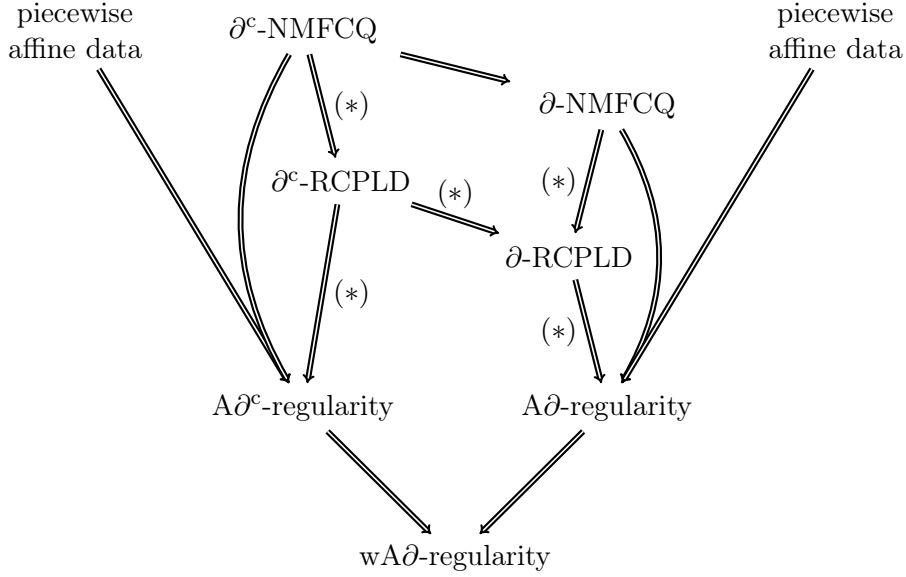

\section{Application to mathematical programs with complementarity constraints}\label{sec:MPCCs}

Here, we apply our results from \cref{sec:asymptotic_concepts} to complementarity-constrained
optimization problems and compare our 
findings with the ones from
\cite{AndreaniHaeserSecchinSilva2019,Ramos2019}. We focus on stationarity conditions
and constraint qualifications. 

For continuously differentiable data functions $f\colon\R^n\to\R$,
$g_i\colon\R^n\to\R$ ($i\in I^g:=\{1,\ldots,m^g\}$), $h_i\colon\R^n\to\R$ ($i\in I^h:=\{1,\ldots,m^h\}$),
and $G_i,H_i\colon\R^n\to\R$ ($i\in I^\textup{cc}:=\{1,\ldots,m^\textup{cc}\}$),
where  $m^g,m^h\in\N\cup\{0\}$ 
and $m^\textup{cc}\in\N$ are arbitrary natural numbers, we consider the so-called
\emph{mathematical program with complementarity constraints} given by
\begin{equation}\label{eq:MPCC}\tag{MPCC}
	\min\left\{
		f(z)
		\,\middle|\,
		\begin{aligned}
		&g_i(z) \leq 0\,(i\in I^g),\,
		h_i(z) = 0\,(i\in I^h),\\
		&0 \leq G_i(z) \perp H_i(z) \geq 0\,(i\in I^\textup{cc})
		\end{aligned}
	\right\}.
\end{equation}
The latter has been studied quite intensively from the viewpoint of theory and numerical
practice since this model covers numerous interesting real-world applications while being
inherently irregular due to the challenging combinatorial structure of the feasible set,
see e.g.\ \cite{LuoPangRalph1996,OutrataKocvaraZowe1998,ScheelScholtes2000,Ye2005} 
and the references therein.
Throughout the section, the feasible set of \eqref{eq:MPCC} will be denoted by $\mathcal Z^\textup{cc}$.
For brevity of notation, let $g\colon\R^n\to\R^{m^g}$, $h\colon\R^n\to\R^{m^h}$, and
$G,H\colon\R^n\to\R^{m^{\textup{cc}}}$ be the mappings which possess the component functions
$g_i$ ($i\in I^g$), $h_i$ ($i\in I^h$), $G_i$ ($i\in I^{\textup{cc}}$), and
$H_i$ ($i\in I^{{\textup{cc}}}$), respectively. 
Furthermore, for multipliers $\lambda^{g}\in\R^{m^g}$, $\lambda^h\in\R^{m^h}$, 
and $\lambda^{G},\lambda^{H}\in\R^{m^\textup{cc}}$, we exploit the MPCC-Lagrangians
\begin{align*}
	\forall z\in\R^n\colon\quad
	\mathcal L^\textup{cc}_0(z,\lambda^g,\lambda^h,\lambda^G,\lambda^H)
	&:=
	(\lambda^g)^\top g(z)+(\lambda^h)^\top h(z)+(\lambda^G)^\top G(z)+(\lambda^H)^\top H(z),\\
	\mathcal L^\textup{cc}(z,\lambda^g,\lambda^h,\lambda^G,\lambda^H)
	&:=
	f(z)+\mathcal L^\textup{cc}_0(z,\lambda^g,\lambda^h,\lambda^G,\lambda^H).
\end{align*}
For an arbitrary feasible point $\bar z\in\mathcal Z^\textup{cc}$, we define the following well-known
index sets:
\begin{align*}
	I^{+0}(\bar z)&:=\{i\in I^\textup{cc}\,|\,G_i(\bar z)>0\,\land\,H_i(\bar z)=0\},\\
	I^{0+}(\bar z)&:=\{i\in I^\textup{cc}\,|\,G_i(\bar z)=0\,\land\,H_i(\bar z)>0\},\\
	I^{00}(\bar z)&:=\{i\in I^\textup{cc}\,|\,G_i(\bar z)=H_i(\bar z)=0\}.
\end{align*}

In order to apply the theory from \cref{sec:asymptotic_concepts} to the model \eqref{eq:MPCC},
we reformulate the final $m^\textup{cc}$ so-called complementarity constraints with the aid
of the minimum function $\varphi_\textup{min}\colon\R^2\to\R$ given by
$\varphi_\textup{min}(a,b):=\min(a,b)$ for all $(a,b)\in\R^2$ in the following way:
\begin{equation}\label{eq:min_ref_compl}
	\varphi_\textup{min}(G_i(z),H_i(z))\,=\,0\quad(i\in I^\textup{cc}).
\end{equation}
Indeed, the resulting problem is equivalent to \eqref{eq:MPCC} since $\varphi_\textup{min}$
is a so-called \emph{nonlinear complementarity problem function} (NCP-function for short), 
see \cite{Galantai2012,KanzowYamashitaFukushima1997} for an overview. 
In order to evaluate the asymptotic stationarity and regularity conditions from
\cref{sec:asymptotic_concepts} for \eqref{eq:MPCC}, 
based on the reformulation \eqref{eq:min_ref_compl} of the complementarity
constraints, we exploit \cref{lem:composed_subdifferential_of_minimum_function}.

\begin{proposition}\label{prop:asymptotic_stationarity_MPCC}
	Let $\bar z\in\mathcal Z^\textup{cc}$ be a feasible point of \eqref{eq:MPCC}.
	Then the following assertions hold.
	\begin{enumerate}
		\item \label{item:asym_CStat_MPCC}
			The point $\bar z$ is A$\cdiff$-stationary for \eqref{eq:MPCC}
			using the reformulation \eqref{eq:min_ref_compl} if and only if there are
			sequences $\{z^k\}_{k\in\N},\{\varepsilon^k\}_{k\in\N}\subset\R^n$, 
			$\{\lambda^{g,k}\}_{k\in\N}\subset\R^{m^g}$, $\{\lambda^{h,k}\}_{k\in\N}\subset\R^{m^h}$
			and $\{\lambda^{G,k}\}_{k\in\N},\{\lambda^{H,k}\}_{k\in\N}\subset\R^{m^{\textup{cc}}}$
			such that $z^k\to\bar z$, $\varepsilon^k\to 0$, and
			\begin{subequations}\label{eq:asym_CStat_MPCC}
				\begin{align}
				\label{eq:asym_CStat_MPCC_der}
					&\forall k\in\N\colon\quad
					\varepsilon^k
					=
					\nabla_z\mathcal L^\textup{cc}
						(z^k,\lambda^{g,k},\lambda^{h,k},\lambda^{G,k},\lambda^{H,k}),\\
				\label{eq:asym_CStat_MPCC_compl_slack}
					&\forall k\in\N,\,\forall i\in I^g\colon\quad
						\min(\lambda_i^{g,k},-g_i(\bar z))=0,\\
				\label{eq:asym_CStat_MPCC_I+0}
					&\forall k\in\N,\,\forall i\in I^{+0}(\bar z)\colon\quad 
						\lambda_i^{G,k}=0,\\
				\label{eq:asym_CStat_MPCC_I0+}
					&\forall k\in\N,\,\forall i\in I^{0+}(\bar z)\colon\quad 
						\lambda_i^{H,k}=0,\\
				\label{eq:asym_CStat_MPCC_I00}
					&\forall k\in\N,\,\forall i\in I^{00}(\bar z)\colon\quad
						\begin{aligned}
							&G_i(z^k)<H_i(z^k)&&\Longrightarrow&&\lambda^{H,k}_i=0,&\\
							&G_i(z^k)>H_i(z^k)&&\Longrightarrow&&\lambda^{G,k}_i=0,&\\
							&G_i(z^k)=H_i(z^k)&&\Longrightarrow&
								&(\lambda^{G,k}_i)(\lambda^{H,k}_i)\geq 0.&
						\end{aligned}
				\end{align}
			\end{subequations}
			Particularly, we find
			\begin{equation}
				\label{eq:asym_CStat_MPCC_I00_weak}
					\forall k\in\N,\,\forall i\in I^{00}(\bar z)\colon\quad
						(\lambda^{G,k}_i)(\lambda^{H,k}_i)\geq 0
			\end{equation}
			in this situation.
		\item \label{item:asym_MStat_MPCC}
			The point $\bar z$ is A$\partial$-stationary for \eqref{eq:MPCC}
			using the reformulation \eqref{eq:min_ref_compl} if and only if there are
			sequences $\{z^k\}_{k\in\N},\{\varepsilon^k\}_{k\in\N}\subset\R^n$, 
			$\{\lambda^{g,k}\}_{k\in\N}\subset\R^{m^g}$, $\{\lambda^{h,k}\}_{k\in\N}\subset\R^{m^h}$
			and $\{\lambda^{G,k}\}_{k\in\N},\{\lambda^{H,k}\}_{k\in\N}\subset\R^{m^{\textup{cc}}}$
			such that $z^k\to\bar z$, $\varepsilon^k\to 0$, 
			\eqref{eq:asym_CStat_MPCC_der}-\eqref{eq:asym_CStat_MPCC_I0+}, and
			\[
				\forall k\in\N,\,\forall i\in I^{00}(\bar z)\colon\quad
						\begin{aligned}
							&G_i(z^k)<H_i(z^k)&&\Longrightarrow&&\lambda^{H,k}_i=0,&\\
							&G_i(z^k)>H_i(z^k)&&\Longrightarrow&&\lambda^{G,k}_i=0,&\\
							&G_i(z^k)=H_i(z^k)&&\Longrightarrow&
								&\lambda^{G,k}_i,\,\lambda^{H,k}_i<0\\
							& && &&\qquad
								\lor\,(\lambda^{G,k}_i)(\lambda^{H,k}_i)=0.&
						\end{aligned}
			\]
			Particularly, we find
			\begin{equation}
				\label{eq:asym_MStat_MPCC_I00_weak}
					\forall k\in\N,\,\forall i\in I^{00}(\bar z)\colon\quad
						\lambda^{G,k}_i,\,\lambda^{H,k}_i<0\,\lor\,
								(\lambda^{G,k}_i)(\lambda^{H,k}_i)=0
			\end{equation}
			in this situation. Above, $\lor$ denotes the logical `or'.
	\end{enumerate}
\end{proposition}
\begin{proof}
	This follows by applying \cref{def:asymptotic_stationarity} to the reformulated
	problem \eqref{eq:MPCC} while respecting \cref{lem:composed_subdifferential_of_minimum_function}.
	On the route, one has to observe that for each sequence $\{z^k\}_{k\in\N}\subset\R^n$ which
	satisfies $z^k\to\bar z$, continuity of $G$ and $H$ yields the inclusions
	$I^{0+}(\bar z)\subset\{i\in I^{\textup{cc}}\,|\,G_i(z^k)<H_i(z^k)\}$ and
	$I^{+0}(\bar z)\subset\{i\in I^\textup{cc}\,|\,G_i(z^k)>H_i(z^k)\}$ for large enough $k\in\N$.
\end{proof}

Let us point out that our choice to use the NCP-function $\varphi_\textup{min}$ for the reformulation
of the complementarity constraints is \emph{essential} in order to guarantee validity of
\cref{prop:asymptotic_stationarity_MPCC}. Exemplary, let us mention that strictly weaker asymptotic
stationarity conditions would be obtained if one exploits the famous Fischer--Burmeister function
or the Kanzow--Schwartz function for that purpose. One nearby reason behind this fact
is that these NCP-functions are actually \emph{too} smooth.
A related observation has been made in \cite[Section~3]{Mehlitz2020} where or-constraints have been
reformulated with the aid of so-called \emph{or-compatible} NCP-functions.

In the light of \cref{prop:asymptotic_stationarity_MPCC}, the following definition is reasonable.
\begin{definition}\label{def:asym_stat_MPCC}
	Let $\bar z\in\mathcal Z^\textup{cc}$ be a feasible point of \eqref{eq:MPCC}.
	\begin{enumerate}
		\item We call $\bar z$ asymptotically MPCC-$\cdiff$-stationary 
			(MPCC-A$\cdiff$-stationary) if there are sequences
			$\{z^k\}_{k\in\N},\{\varepsilon^k\}_{k\in\N}\subset\R^n$, 
			$\{\lambda^{g,k}\}_{k\in\N}\subset\R^{m^g}$, $\{\lambda^{h,k}\}_{k\in\N}\subset\R^{m^h}$
			as well as $\{\lambda^{G,k}\}_{k\in\N},\{\lambda^{H,k}\}_{k\in\N}\subset\R^{m^{\textup{cc}}}$
			such that $z^k\to\bar z$, $\varepsilon^k\to 0$, 
			\eqref{eq:asym_CStat_MPCC_der}-\eqref{eq:asym_CStat_MPCC_I0+}, and
			\eqref{eq:asym_CStat_MPCC_I00_weak} hold.
		\item We call $\bar z$ asymptotically MPCC-$\partial$-stationary
			(MPCC-A$\partial$-stationary) if there are sequences
			$\{z^k\}_{k\in\N},\{\varepsilon^k\}_{k\in\N}\subset\R^n$, 
			$\{\lambda^{g,k}\}_{k\in\N}\subset\R^{m^g}$, $\{\lambda^{h,k}\}_{k\in\N}\subset\R^{m^h}$
			as well as $\{\lambda^{G,k}\}_{k\in\N},\{\lambda^{H,k}\}_{k\in\N}\subset\R^{m^{\textup{cc}}}$
			such that $z^k\to\bar z$, $\varepsilon^k\to 0$, 
			\eqref{eq:asym_CStat_MPCC_der}-\eqref{eq:asym_CStat_MPCC_I0+}, and
			\eqref{eq:asym_MStat_MPCC_I00_weak} hold.
	\end{enumerate}
\end{definition}

Observe that these stationarity conditions generalize the classical notions of C- and
M-stationarity of \eqref{eq:MPCC} which are used throughout the literature, 
see e.g.\ \cite[Section~2.2]{Ye2005} for the precise definitions.

Let us compare the conditions from \cref{def:asym_stat_MPCC} 
with the sequential stationarity notions for \eqref{eq:MPCC}
which were studied in 
\cite{AndreaniHaeserSecchinSilva2019} and \cite{Ramos2019}, respectively.
First, we focus on the stationarity conditions w.r.t.\ the limiting subdifferential.
Therefore, we introduce a set 
$\Lambda^\textup{cc}:=\{(a,b)\in\R^2\,|\,0\leq a\perp b\geq 0\}$ and observe
that the conditions \eqref{eq:asym_CStat_MPCC_I+0}, \eqref{eq:asym_CStat_MPCC_I0+}, and
\eqref{eq:asym_MStat_MPCC_I00_weak} are equivalent to
\[
	\forall k\in\N\,\forall i\in I^\textup{cc}\colon\quad
	(\lambda^{G,k}_i,\lambda^{H,k}_i)\in\mathcal N_{\Lambda^{\textup{cc}}}(G_i(\bar z),H_i(\bar z)).
\]
Thus, MPCC-A$\partial$-stationary points are MPCC-AKKT points in the sense of
\cite[Definition~3.2]{Ramos2019}. Exploiting the inclusion 
$\mathcal N_{\Lambda^{\textup{cc}}}(a,b)\subset\mathcal N_{\Lambda^\textup{cc}}(\bar a,\bar b)$, which,
due to the disjunctive structure of $\Lambda^\textup{cc}$,
holds for all $(\bar a,\bar b)\in\Lambda^\textup{cc}$ and all $(a,b)\in\Lambda^\textup{cc}\cap U$
where $U\subset\R^2$ is a sufficiently small neighborhood of $(\bar a,\bar b)$, the converse relation
holds true as well. Thus, the comments at the end of 
\cite[Section~3]{AndreaniHaeserSecchinSilva2019} justify that  
MPCC-A$\partial$-stationarity corresponds to AM-stationarity in the sense of 
\cite[Definition~3.3]{AndreaniHaeserSecchinSilva2019}.
In similar way, we see that MPCC-A$\cdiff$-stationarity corresponds to
AC-stationarity from \cite[Definition~3.3]{AndreaniHaeserSecchinSilva2019}.
Our approach, thus, recovers the sequential stationarity concepts from
\cite{AndreaniHaeserSecchinSilva2019,Ramos2019}.

Clearly, due to \cref{thm:local_minimizers_are_asymptotically_stationary}, each local minimizer
of \eqref{eq:MPCC} is MPCC-A$\partial$-stationary. 
Observing that MPCC-A$\gdiff$-stationarity is slightly weaker than 
A$\gdiff$-stationarity for \eqref{eq:MPCC} based on the reformulation
\eqref{eq:min_ref_compl} of the complementarity constraints, 
we need a slightly stronger constraint qualification than
A$\gdiff$-regularity in order to infer validity of the classical C- or M-stationarity conditions.
More precisely, for some fixed feasible point $\bar z\in\mathcal Z^\textup{cc}$, 
one would be tempted to postulate outer semicontinuity of the mappings
$\MM^{\textup{c}}_\textup{cc},\MM_{\textup{cc}}\colon\R^n\tto\R^n$ given below for each
$z\in\R^n$:
\begin{align*}
	\MM^{\textup{c}}_\textup{cc}(z)
	&:=
	\left\{
		\nabla_z\mathcal L^\textup{cc}_0(z,\lambda^g,\lambda^h,\lambda^G,\lambda^H)
		\,\middle|\,
		\begin{aligned}
			&\forall i\in I^g\colon\,\min(\lambda^g_i,-g_i(\bar z))=0,\\
			&\forall i\in I^{+0}(\bar z)\colon\,\lambda^G_i=0,\\
			&\forall i\in I^{0+}(\bar z)\colon\,\lambda^H_i=0,\\
			&\forall i\in I^{00}(\bar z)\colon\,(\lambda^G_i)(\lambda^H_i)\geq 0
		\end{aligned}
	\right\},\\
	\MM_{\textup{cc}}(z)
	&:=
	\left\{
		\nabla_z\mathcal L^\textup{cc}_0(z,\lambda^g,\lambda^h,\lambda^G,\lambda^H)
		\,\middle|\,
		\begin{aligned}
			&\forall i\in I^g\colon\,\min(\lambda^g_i,-g_i(\bar z))=0,\\
			&\forall i\in I^{+0}(\bar z)\colon\,\lambda^G_i=0,\\
			&\forall i\in I^{0+}(\bar z)\colon\,\lambda^H_i=0,\\
			&\forall i\in I^{00}(\bar z)\colon\,\lambda^G_i,\lambda^H_i<0\,\lor\,(\lambda^G_i)(\lambda^H_i)=0
		\end{aligned}
	\right\}.
\end{align*}
Observe that this precisely recovers the notions of AC- and AM-regularity from
\cite[Definition~4.1]{AndreaniHaeserSecchinSilva2019}. 
Employing \cite[Theorem~5.3]{Mehlitz2020b}, outer semicontinuity of $\MM_{\textup{cc}}$
equals the concept of AM-regularity from \cite[Definition~3.8]{Mehlitz2020b} applied
to \eqref{eq:MPCC} where the complementarity constraints are reformulated as
\[
	(G_i(z),H_i(z))\in\Lambda_\textup{cc}\;(i\in I^{\textup{cc}}),
\]
and this, finally, is the same as the condition MPCC-CCP introduced in
\cite[Definition~3.9]{Ramos2019}. Note that the condition
\[
	\limsup\limits_{z\to\bar z}\MM_{\textup{cc}}(z)
	\subset
	\MM_{\textup{cc}}^\textup{c}(\bar z),
\]
which, for example, holds in \cref{ex:AC_but_not_AM_regular}, amounts to a new
constraint qualification for \eqref{eq:MPCC} which guarantees C-stationarity of
local minimizers. It is motivated by the more general concept of wA$\partial$-regularity
from \cref{def:asymptotic_regularity}.

Some algorithmic benefits of the above MPCC-tailored notions of asymptotic stationarity
and regularity have been envisioned in \cite{AndreaniHaeserSecchinSilva2019,Ramos2019}.
More precisely, let us mention that the convergence analysis associated with some
MPCC-tailored penalty, multiplier-penalty, and relaxation methods can be carried out 
in the presence of sequential regularity.

The subsequent remark, which closes this section, underlines that the above ideas can
be adapted in order to tackle other problem classes from disjunctive optimization. 

\begin{remark}\label{rem:or_constrained_programming}
For continuously differentiable functions $\tilde G_i,\tilde H_i\colon\R^n\to\R$
($i\in I^\textup{dc}:=\{1,\ldots,m^{\textup{dc}}\}$) where $m^\textup{dc}\in\N$ is an
arbitrary natural number, the model
\begin{equation}\label{eq:MPOC}\tag{MPOC}
	\min\left\{
		f(z)
		\,\middle|\,
		\begin{aligned}
		&g_i(z) \leq 0\,(i\in I^g),\,
		h_i(z) = 0\,(i\in I^h),\\
		&\tilde G_i(z)\leq 0\,\lor\,\tilde H_i(z) \leq 0\,(i\in I^\textup{dc})
		\end{aligned}
	\right\}
\end{equation}
is referred to as a \emph{mathematical program with or-constraints} in the literature,
see \cite{Mehlitz2019,Mehlitz2020}.
It basically suffers from the same problems as \eqref{eq:MPCC} which is why weak 
stationarity notions and constraint qualifications as well as problem-tailored solution
methods have been studied for \eqref{eq:MPOC}. Reformulating the final $m^\textup{dc}$
so-called or-constraints with the aid of 
\[
	\varphi_\textup{min}(\tilde G_i(z),\tilde H_i(z))\leq 0\quad (i\in I^\textup{dc}),
\]
which also has been suggested in \cite[Section~3]{Mehlitz2020}, we can apply the
theory of \cref{sec:asymptotic_concepts} similarly as above in order to derive
asymptotic stationarity and regularity conditions for \eqref{eq:MPOC}.
More precisely, the resulting concepts of  A$\cdiff$- and A$\partial$-stationarity
generalize the notions of M- and weak stationarity for \eqref{eq:MPOC} which have been
introduced in \cite[Definition~7.1]{Mehlitz2019}.

Replacing the final $m^\textup{dc}$ constraints of \eqref{eq:MPOC} by
\begin{equation}\label{eq:VC}\tag{VC}
	\tilde H_i(z)\geq 0,\quad\tilde G_i(z)\,\tilde H_i(z)\leq 0\quad (i\in I^\textup{dc}),
\end{equation}
a so-called \emph{vanishing-constrained} optimization problem is obtained,
see e.g.\ \cite{AchtzigerKanzow2008,HoheiselKanzow2007,HoheiselPablosPooladianSchwartzSteverango2020}
and the references therein for an overview. 
Introducing the Lipschitz continuous function $\varphi_\textup{vc}\colon\R^2\to\R$ by means of
\[
	\forall (a,b)\in\R^2\colon\quad
	\varphi_\textup{vc}(a,b):=
	\begin{cases}
		a	& 	a\geq 0,\,b\geq a,\\
		0	&	a<0,\,b\geq 0,\\
		|b|	&	\text{otherwise,}
	\end{cases}
\]
one can easily check that the constraint system \eqref{eq:VC} is equivalent to
\[
	\varphi_\textup{vc}(\tilde G_i(z),\tilde H_i(z))\leq 0\quad (i\in I^\textup{dc}).
\]
Exploiting the rules of subdifferential calculus, one can check that $\cdiff$- and $\partial$-stationarity
of the associated problem \eqref{eq:Lipschitzian_program} recover the M- and weak stationarity system
of the vanishing-constrained optimization problem, 
see \cite[Definition~2.4]{HoheiselPablosPooladianSchwartzSteverango2020} for the precise definitions.
Consequently, we can proceed as above in order to derive reasonable notions of asymptotic stationarity
and regularity for vanishing-constrained optimization problems.

Deriving the details in both cases is left to the interested reader.
\end{remark}

Finally, let us note that the results of this section can be extended to disjunctive programs 
with nonsmooth data functions, see e.g.\ \cite{KazemiKanzi2018,MovahedianNobakhtian2009,MovahedianNobakhtian2009b}
for applications, stationarity conditions, and constraint qualifications associated with nonsmooth complementarity-
and vanishing-constrained optimization problems.
We would like to point the reader's attention to the fact that reformulating nonsmooth complementarity
constraints with the aid of $\varphi_{\min}$ yields equality constraints of type \eqref{eq:min_ref_compl} again,
but the subdifferentials of the mappings $z\mapsto\varphi_{\min}(G_i(z),H_i(z))$ ($i\in I^{\textup{cc}}$) are now
much more difficult to evaluate in general since $G_i$ and $H_i$ are potentially nonsmooth. 
Most likely, one might only be in position to compute upper estimates of these subdifferentials with the aid
of suitable chain rules (\cref{lem:composed_subdifferential_of_minimum_function} does not apply anymore), 
i.e., a result similar to \cref{prop:asymptotic_stationarity_MPCC} would yield slightly
weaker conditions than actual A$\cdiff$- and A$\partial$-stationarity. Similar issues are likely to pop up
when considering other types of nonsmooth disjunctive constraints.

\section{Application to bilevel optimization problems with affine constraints}\label{sec:bilevel}

In this section, we suggest a solution method for the numerical handling of the
optimistic bilevel optimization problem
\begin{equation}\label{eq:upper_level}\tag{BPP}
	\min\limits_{x,y}\{f(x,y)\,|\,Cx+Dy\leq d,\,y\in\Psi(x)\}
\end{equation}
where $f\colon\R^n\times\R^m\to\R$ is continuously differentiable
and $\Psi\colon\R^n\tto\R^m$ is the solution mapping of the fully linear
parametric optimization problem
\begin{equation}\label{eq:lower_level}\tag{P$(x)$}
	\min\limits_y\{c^\top y\,|\,Ax+By\leq b\},
\end{equation}
i.e., for each $x\in\R^n$, $\Psi$ assigns to $x$ the (possibly empty) solution set
$\Psi(x)$ of \eqref{eq:lower_level}.
Above, $A\in\R^{p\times n}$, $B\in\R^{p\times m}$, $C\in\R^{q\times n}$, $D\in\R^{q\times m}$,
$b\in\R^p$, $c\in\R^m$, and $d\in\R^q$ are fixed matrices.
We refer to \eqref{eq:upper_level} and \eqref{eq:lower_level} as the upper and lower
level problem, respectively. Although the constraints $Cx+Dy\leq d$ as well as the lower
level problem \eqref{eq:lower_level} are fully affine, \eqref{eq:upper_level} possesses
a nonconvex feasible set which is only implicitly given. This makes \eqref{eq:upper_level}
notoriously difficult even if $f$ is fully linear.  Besides, the model \eqref{eq:upper_level} 
covers numerous interesting applications such as inverse linear programming, balancing in
energy and traffic networks, or data compression, 
see \cite{Dempe2002,Dempe2020,DempeKalashnikovPerezValdesKalashnykova2015} and the references therein
for an introduction to and a satisfying overview of bilevel optimization.

Subsequently, we want to exploit the so-called optimal value or marginal function 
$\vartheta\colon\R^n\to\overline{\R}$ of \eqref{eq:lower_level} given by
\[
	\forall x\in\R^n\colon\quad
	\vartheta(x):=\inf\limits_y\{c^\top y\,|\,Ax+By\leq b\}.
\]
Due to full linearity of \eqref{eq:lower_level}, $\vartheta$ is a convex and piecewise affine
function. A fully explicit formula for its subdifferential can be found e.g.\ in 
\cite[Proposition~4.1]{YeWu2008}.
\begin{lemma}\label{lem:subdiff_of_theta}
	For each $\bar x\in\dom\vartheta$, the formula
	\[
		\partial\vartheta(\bar x)=A^\top S(\bar x)
	\]
	holds where $S(\bar x)$ is the solution set of the dual problem associated with
	\hyperref[eq:lower_level]{\textup{(P$(\bar x)$)}} given by
	\[
		S(\bar x):=\argmax\limits_\lambda\{(A\bar x-b)^\top\lambda\,|\,B^\top\lambda=-c,\,\lambda\geq 0\}.
	\]
\end{lemma}

In this section, we exploit the well-known observation that \eqref{eq:upper_level} is
equivalent to
\begin{equation}\label{eq:VFR}\tag{VFR}
	\min\limits_{x,y}\{f(x,y)\,|\,Ax+By\leq b,\,c^\top y-\vartheta(x)\leq 0,\,Cx+Dy\leq d\}
\end{equation}
which is referred to as the \emph{value function reformulation} of \eqref{eq:upper_level}.
Observing that the estimate $c^\top y-\vartheta(x)\geq 0$ holds for each pair $(x,y)\in\R^n\times\R^m$
which satisfies $Ax+By\leq b$, the constraint $c^\top y-\vartheta(x)\leq 0$ is actually active 
at each feasible point of \eqref{eq:VFR}. One can check that this causes that, exemplary,
$\cdiff$-NMFCQ cannot hold at the feasible points of \eqref{eq:VFR}, see e.g.\
\cite[Proposition~3.2]{YeZhu1995}. Besides, $\vartheta$ is
an implicitly given function whose full computation is not possible as soon as practically relevant
problems are under consideration. Nevertheless, starting with 
\cite{Outrata1988}, \eqref{eq:VFR} has been used successfully 
in the literature in order to derive optimality conditions and solution methods for 
\eqref{eq:upper_level}.

\subsection{Stationarity conditions}\label{sec:stationarity_conditions}

Noting that the optimal value function $\vartheta$ is locally Lipschitzian 
at each point from $\intr\dom\vartheta$,
the following result is an immediate consequence of \cref{lem:affine_constraints_AC_regular}.
\begin{proposition}\label{prop:stationarity_system_BPP}
	Let $(\bar x,\bar y)\in\R^n\times\R^m$ be a local minimizer of \eqref{eq:upper_level}
	and assume that $\bar x\in\intr\dom\vartheta$ holds.	
	Then we find multipliers $\lambda,\lambda'\in\R^p$ and $\mu\in\R^q$ as well
	as $\sigma\geq 0$ such that
	\begin{subequations}\label{eq:stationarity_BPP}
		\begin{align}
			\label{eq:stationarity_PBB_x}
				&0=\nabla_xf(\bar x,\bar y)+A^\top(\lambda-\sigma\lambda')+C^\top\mu,\\
			\label{eq:stationarity_BPP_y}
				&0=\nabla_yf(\bar x,\bar y)+B^\top(\lambda-\sigma\lambda')+D^\top\mu,\\
			\label{eq:stationarity_BPP_compl_slack_ll}
				&0\leq\lambda\perp b-A\bar x-B\bar y,\\
			\label{eq:stationarity_BPP_compl_slack_ul}
				&0\leq\mu\perp d-C\bar x-D\bar y,\\
			\label{eq:stationarity_BPP_subdiff}
				&0=c+B^\top\lambda',\quad 0\leq\lambda'\perp b-A\bar x-B\bar y.
		\end{align}
	\end{subequations}
\end{proposition}
\begin{proof}
	By assumption, $(\bar x,\bar y)$ is a local minimizer of \eqref{eq:VFR} which is
	a Lipschitzian optimization problem in some neighborhood of this point.
	Furthermore, the constraints of \eqref{eq:VFR} are given in terms of piecewise affine
	data functions. Applying \cref{thm:AC_regular_local_minimizers} and
	\cref{lem:affine_constraints_AC_regular}, $(\bar x,\bar y)$ is a $\cdiff$-stationary
	point of \eqref{eq:VFR}. Thus, we find multipliers $\lambda\in\R^p$, $\mu\in\R^q$, and
	$\sigma\geq 0$ as well as $\xi\in\cdiff(-\vartheta)(\bar x)$ satisfying
	\begin{equation}\label{eq:surrugate_stationarity}
		\begin{aligned}
		&0=\nabla_xf(\bar x,\bar y)+A^\top\lambda+\sigma\xi+C^\top\mu,\\
		&0=\nabla_yf(\bar x,\bar y)+B^\top\lambda+\sigma c+D^\top\mu,
		\end{aligned}
	\end{equation}
	and \eqref{eq:stationarity_BPP_compl_slack_ll} as well as \eqref{eq:stationarity_BPP_compl_slack_ul}.
	Observe that we have $\cdiff(-\vartheta)(\bar x)=-\cdiff\vartheta(\bar x)=-\partial\vartheta(\bar x)$
	by convexity of $\vartheta$.
	Due to \cref{lem:subdiff_of_theta} and strong duality of linear programming,
	we find $\lambda'\in\R^p$ satisfying \eqref{eq:stationarity_BPP_subdiff}
	and $\xi=-A^\top\lambda'$. Inserting this and $c=-B^\top\lambda'$ into
	\eqref{eq:surrugate_stationarity} yields the claim.
\end{proof}

Let us point out some important facts regarding the above result. First, under the assumptions made,
$(\bar x,\bar y)$ is already a $\partial$-stationary point of \eqref{eq:VFR}. 
However, in order to evaluate $\partial(-\vartheta)(\bar x)$ successfully via
\cref{lem:subdiff_of_theta}, we have to estimate this set from above by $\cdiff(-\vartheta)(\bar x)$,
and, again, end up with $\cdiff$-stationarity. 
This observation already has been made in \cite[Sections~3 and~4]{DempeDuttaMordukhovich2007}
where a more general problem class has been considered.

Note that although the proof of \cref{prop:stationarity_system_BPP} via
\cref{thm:AC_regular_local_minimizers} is novel, the result is well known in the
literature on bilevel optimization. Indeed, the special structure of the lower level
problem \eqref{eq:lower_level} implies that \eqref{eq:VFR} is so-called \emph{partially calm}
at $(\bar x,\bar y)$, 
see \cite[Definition~3.1]{YeZhu1995} and \cite[Theorem~4.1]{MehlitzMinchenkoZemkoho2020} for details,
which equivalently means that there is some $\sigma\geq 0$ such that $(\bar x,\bar y)$ is a
local minimizer of
\[
	\min\limits_{x,y}\{f(x,y)+\sigma(c^\top y-\vartheta(x))\,|\,Ax+By\leq b,\,Cx+Dy\leq d\},
\]
and this problem can be tackled with standard KKT-theory since its constraints are affine.
This approach precisely recovers the stationarity system \eqref{eq:stationarity_BPP}.
Related approaches have been used e.g.\ in
\cite{DempeDuttaMordukhovich2007,DempeZemkoho2013,MordukhovichNamPhan2012} in order to derive
necessary optimality conditions for more general bilevel optimization problems.
Finally, we refer the interested reader to \cite[Corollary~4.1]{Ye2004} where yet another proof of
\cref{prop:stationarity_system_BPP} can be found which does not rely on the concept of 
partial calmness as well. 

Note that the concept of asymptotic regularity is also applicable to more general bilevel
optimization problems via their optimal value reformulation as long as the associated
lower level optimal value function is locally Lipschitz continuous in a neighborhood of the
reference point. Keeping the weakness of asymptotic regularity in mind, this approach could
be suitable to find constraint qualifications which actually apply in a reasonable way to
bilevel optimization problems. Further potentially promising constraint qualifications 
can be obtained when
applying these concepts to the combined reformulation of the bilevel optimization problem where
lower level value function and necessary optimality conditions are used in parallel, see
\cite{YeZhu2010}. A detailed investigation of these approaches is, however, beyond this paper's scope.

\subsection{A penalty-DC-method and its convergence properties}\label{sec:theory_DC}

In the remainder of this section, we assume that $f=f_1-f_2$ holds where the functions
$f_1,f_2\colon\R^n\times\R^m\to\R$ are continuously differentiable and convex, i.e.,
that $f$ is a so-called DC-function, see \cite{HorstThoai1999,ThiDinh2018} for an overview of
DC-optimization.
This implies that \eqref{eq:VFR} is a DC-problem, i.e., its objective function as well as its
constraints are DC-functions.
Below, we suggest a simple algorithm for the numerical solution of \eqref{eq:upper_level}
which makes use of this observation while utilizing that the nonsmoothness is
encapsulated only in $\vartheta$. Let us mention that this structural properties already 
have been used partially in the recent paper \cite{YeYuanZengZhang2021} for an algorithmic treatment 
of more general bilevel optimization problems where the data is allowed to be nonsmooth. 
More specifically and applied to our setting, the authors investigated the surrogate
problem
\begin{equation}\label{eq:relaxed_OVR}\tag{VFR$(\varepsilon)$}
		\min\limits_{x,y}\{f(x,y)\,|\,Ax+By\leq b,\,c^\top y-\vartheta(x)\leq\varepsilon,\,Cx+Dy\leq d\},
\end{equation}
for some \emph{fixed} relaxation parameter $\varepsilon\geq 0$, which is iteratively
solved with some inexact DC-methods. In case $\varepsilon>0$, a convergence theory is provided
based on the observation that \eqref{eq:relaxed_OVR} behaves regular in this case. However, the
feasible set of \eqref{eq:relaxed_OVR} might be essentially larger than the one of \eqref{eq:VFR}
which is why this method most likely does not compute (asymptotically) feasible points 
of \eqref{eq:upper_level}. Fixing $\varepsilon=0$, no convergence guarantees were provided
in \cite{YeYuanZengZhang2021}. We note that the authors of this paper did not investigate 
the situation where the relaxation parameter $\varepsilon$ is iteratively sent towards $0$.
We would like to mention
\cite{HoaiTaoNguyenVan2008}
where DC-programming has been used to solve bilevel optimization problems 
of special structure to global optimality.
However, in the latter paper, the authors do not exploit convexity of the optimal value function
but make use of lower level optimality conditions to proceed. 
Yet another related approach for the global solution of bilevel optimization problems with fully
convex lower level data can be found in \cite{DempeHarderMehlitzWachsmuth2019,DempeFranke2016}.

Here, we stick a different path and investigate a \emph{penalty} approach w.r.t.\ the
constraint $c^\top y-\vartheta(x)\leq 0$ where the resulting subproblems are
solved with the aid of a DC-method. 
More precisely, exploiting the fact that the resulting penalized surrogate problems only possess
affine constraints, they can be solved with the aid of the recently developed boosted DC-method from \cite{AragonArtachoCampoyVuong2020} in reasonable time.
This method basically computes $\cdiff$-stationary points of DC-problems of type
\begin{equation}\label{eq:DC_problem}
	\min\limits_w\{g(w)-h(w)\,|\,\tilde A w\leq \tilde b\}
\end{equation}
where $g,h\colon\R^s\to\R$ are convex functions such that $g$ is continuously differentiable
and $\tilde A\in\R^{t\times s}$ as well as $\tilde b\in\R^t$ are fixed matrices. 
For some fixed iterate $w^k\in\R^s$ and some subgradient $\xi^k\in\partial h(w^k)$, one
identifies a minimizer $z^k\in\R^n$ of the partially linearized \emph{convex} subproblem
\[
	\min\{g(w)-(\xi^k)^\top w\,|\,\tilde Aw\leq\tilde b\}
\]
and performs a suitable \emph{line search} along the direction $z^k-w^k$. 
For step size $\alpha^k$, the new iterate is set to $w^{k+1}:=z^k+\alpha^k(z^k-w^k)$.
The line search is responsible for the speed up. For details, we refer the interested reader to
\cite{AragonArtachoCampoyVuong2020}. Note that setting $w^{k+1}:=z^k$, i.e., choosing step size
$\alpha^k:=0$, recovers the classical DC-method.

For brevity of notation, let us introduce $\mathcal Z_\ell,\mathcal Z_u\subset\R^n\times\R^m$ by means of
\begin{align*}
	\mathcal Z_\ell&:=\{(x,y)\in\R^n\times\R^m\,|\,Ax+By\leq b\},
	\\
	\mathcal Z_u&:=\{(x,y)\in\R^n\times\R^m\,|\,Cx+Dy\leq d\}.
\end{align*}
Throughout the section, we assume that
$\{x\in\R^n\,|\,\exists y\in\R^m\colon\,(x,y)\in \mathcal Z_\ell\}\subset\intr\dom \vartheta$
holds.
This assumption guarantees that $\varphi$ is locally Lipschitz continuous at all \emph{relevant} points.
Furthermore, we assume that $\mathcal Z_u\cap \mathcal Z_\ell$ is nonempty and compact in order to
make sure that \eqref{eq:upper_level} possesses a solution.

In \cref{alg:boosted_DC_penalty}, we state the foreshadowed solution method 
which can be used to tackle \eqref{eq:upper_level}.
\begin{algorithm}
	\begin{description}
		\item [{S0}] Choose $\sigma^0>0$, $\gamma>1$, and 
			$(x^0,y^0)\in \mathcal Z_u\cap \mathcal Z_\ell$.
			Set $k:=0$.
		\item [{S1}] Compute a $\cdiff$-stationary solution 
			$(x^{k+1},y^{k+1})$ of the DC-problem
			\begin{equation}\label{eq:penalized_problem_DC}
				\min\limits_{x,y}\left\{
					f_1(x,y)+\sigma^k\,c^\top y-(f_2(x,y)+\sigma_k\,\vartheta(x))\,\middle|\,
					\begin{aligned}
						&Ax+By\leq b,\\
						&Cx+Dy\leq d
					\end{aligned}
				\right\}
			\end{equation}
			with the aid of the boosted DC-algorithm 
			exploiting the starting point $(x^k,y^k)$.
		\item [{S2}] If $y^{k+1}\in\Psi(x^{k+1})$ holds, STOP. Return $(x^{k+1},y^{k+1})$.
		\item [{S3}] Set $\sigma^{k+1}:=\gamma\,\sigma^k$ as well as $k:=k+1$ and go to \textbf{S1}.
	\end{description}
	\caption{
		Penalty-DC-method for \eqref{eq:upper_level}
		\label{alg:boosted_DC_penalty}
		}
\end{algorithm}
Let us point out the benefits of solving the subproblem \eqref{eq:penalized_problem_DC} with a DC-method.
Clearly, \eqref{eq:penalized_problem_DC} is not given explicitly since $\vartheta$ is an
implicit object whose full computation is not a reasonable option in numerical practice.
However, in each step of the DC-method, the concave part within the objective of
\eqref{eq:penalized_problem_DC} is linearized via subgradient information, and we already
know that subgradients of $\vartheta$ can be computed efficiently by solving a single linear
optimization problem, see \cref{lem:subdiff_of_theta}. 

For our convergence analysis, we assume that \cref{alg:boosted_DC_penalty} produces an
infinite sequence $\{(x^k,y^k)\}_{k\in\N}$. Indeed, if \cref{alg:boosted_DC_penalty}
terminates after finitely many steps, then the final iterate is $\cdiff$-stationary for
\eqref{eq:penalized_problem_DC}. Exploiting \cref{lem:subdiff_of_theta} again, this
already shows that the final iterate is $\cdiff$-stationary for \eqref{eq:VFR}, i.e.,
a point satisfying the stationarity conditions from
\cref{prop:stationarity_system_BPP}.

\begin{lemma}\label{lem:accumulation_points_feasible}
	Each accumulation point $(\bar x,\bar y)\in\R^n\times\R^m$ of
	the sequence $\{(x^k,y^k)\}_{k\in\N}\subset\R^n\times\R^m$ generated by
	\cref{alg:boosted_DC_penalty} is feasible to \eqref{eq:VFR} and, thus, 
	\eqref{eq:upper_level} provided there is a constant
	$\kappa>0$ such that
	\[
		\forall k\in\N\colon\quad
		f(x^{k+1},y^{k+1})+\sigma^k(c^\top y^{k+1}-\vartheta(x^{k+1}))\leq \kappa.
	\]
\end{lemma}
\begin{proof}
	Assume w.l.o.g.\ that $x^k\to\bar x$ and $y^k\to\bar y$ hold.
	By continuity of $f$, the sequence $\{f(x^{k+1},y^{k+1})\}_{k\in\N}$ is bounded.
	Thus, we find $\tilde \kappa>0$ such that
	\begin{align*}
		\forall k\in\N\colon\quad
		0\leq c^\top y^{k+1}-\vartheta(x^{k+1})\leq\tilde \kappa/\sigma_k.
	\end{align*} 
	Taking the limit $k\to\infty$ while exploiting the continuity of all appearing functions
	and $\sigma_k\to\infty$,
	this shows $c^\top\bar y-\vartheta(\bar x)=0$, i.e., $\bar y\in\Psi(\bar x)$.
	This completes the proof.
\end{proof}

Note that the boundedness assumption in \cref{lem:accumulation_points_feasible} 
can be always guaranteed if the subproblems \eqref{eq:penalized_problem_DC} 
are solved to global optimality. 
In this case, one can choose $\kappa:=f(\tilde x,\tilde y)$ as an upper bound of the objective
values where $(\tilde x,\tilde y)\in\R^n\times\R^m$ is an arbitrary feasible point of
\eqref{eq:upper_level}. In this case, each accumulation point of $\{(x^k,y^k)\}_{k\in\N}$ is
already a global minimizer of \eqref{eq:upper_level}.
However, one should note that solving the subproblems globally is, in general, only possible 
by decomposing the domain of $\vartheta$ into its so-called \emph{regions of stability}
where $\vartheta$ behaves in an affine
way, see \cite[Section~4]{DempeFranke2016} for a related approach, 
since this allows to trace back the solution of
\eqref{eq:penalized_problem_DC} to the solution of finitely many convex subproblems if $f_2$
vanishes. Obviously, this approach is already computationally expensive whenever the 
dimension $n$ is of medium size.

Now, we state a convergence result regarding \cref{alg:boosted_DC_penalty}.
\begin{theorem}\label{thm:convergence_to_asymptotically_stationary_points}
	Assume that \cref{alg:boosted_DC_penalty} generates
	a sequence $\{(x^k,y^k)\}_{k\in\N}\subset\R^n\times\R^m$, and let
	$(\bar x,\bar y)\in\R^n\times\R^m$ be an accumulation point of this sequence
	which is feasible to \eqref{eq:upper_level}.
	Then $(\bar x,\bar y)$ is stationary for \eqref{eq:upper_level} in the sense
	of \cref{prop:stationarity_system_BPP}, i.e., there exist multipliers which
	solve the stationarity system \eqref{eq:stationarity_BPP}.
\end{theorem}
\begin{proof}
	Let us assume w.l.o.g.\ that $x^k\to\bar x$ and $y^k\to\bar y$ hold.
	By construction of the method, for each $k\in\N$, we find $\xi^k\in\partial\vartheta(x^{k+1})$,
	$\lambda^k\in\R^p$, and $\mu^k\in\R^q$ which satisfy
	\begin{subequations}\label{eq:approx_C_St_algorithms}
		\begin{align}
				\label{eq:approx_CSt_x}
				&0=\nabla_xf(x^{k+1},y^{k+1})
					-\sigma_k\xi^k
					+A^\top\lambda^k+C^\top\mu^k,\\
				\label{eq:approx_CSt_y}
				&0=\nabla_yf(x^{k+1},y^{k+1})+
					\sigma_k\,c
					+B^\top\lambda^k+D^\top\mu^k,\\
				\label{eq:approx_CSt_compl_g}
				&0\leq\lambda^k\perp b-Ax^{k+1}-By^{k+1},\\
				\label{eq:approx_CSt_compl_Ga}
				&0\leq\mu^k\perp d-Cx^{k+1}-Dy^{k+1}.			
		\end{align}
	\end{subequations}
	Due to $x^k\to\bar x$ and $y^k\to\bar y$ as well as feasibility of
	$(\bar x,\bar y)$ for \eqref{eq:upper_level}, we find that
	$(\bar x,\bar y)$ is A$\cdiff$-stationary for \eqref{eq:VFR}.
	Recalling that each feasible point of \eqref{eq:VFR} is A$\cdiff$-regular,
	$(\bar x,\bar y)$ is already $\cdiff$-stationary for \eqref{eq:VFR}.
	As we have already mentioned before, \eqref{eq:stationarity_BPP} corresponds to
	the $\cdiff$-stationarity system of \eqref{eq:VFR}.
\end{proof}

Let us note that the result of \cref{thm:convergence_to_asymptotically_stationary_points} remains
true whenever the subproblems \eqref{eq:penalized_problem_DC} are only solved up to approximate
$\cdiff$-stationarity in the inner iteration, i.e., 
in terms of the more general DC-problem \eqref{eq:DC_problem}, we
need to have
\[
	\forall k\in\N\colon\quad
	\varepsilon^k\in\nabla g(w^{k+1})-\partial h(w^{k+1})+\tilde A^\top\zeta^k,
	\qquad
	0\leq\zeta^k\perp\tilde b-\tilde Aw^{k+1}
\]
such that $\varepsilon^k\to 0$ is guaranteed as $k\to\infty$. 
Unluckily, the boosted DC-method from \cite{AragonArtachoCampoyVuong2020} does not guarantee
this property of the outputs in general.
Related phenomena motivated the study in \cite{HelouSantosSimoes2020} which, however, comes
along with other drawbacks.

\subsection{Implementation and numerical experiments}

In this section, we are going to provide some numerical results regarding the
computational competitiveness of \cref{alg:boosted_DC_penalty}. Therefore, we
compare the suggested method with two other intuitive penalty approaches which can be used
to solve \eqref{eq:upper_level}. On the one hand, it might be more natural to exploit the
standard DC-algorithm (without boosting) to solve the subproblems \eqref{eq:penalized_problem_DC} in
\cref{alg:boosted_DC_penalty} which leads to a simpler method and speeds up the inner
DC-iterations since no step size computation is necessary. On the other hand, it
might be reasonable to replace \eqref{eq:penalized_problem_DC} by means of
\begin{equation}\label{eq:subproblem_duality_gap}
	\min\limits_{x,y,u}
	\left\{
		f(x,y)+\sigma^k(c^\top y-(Ax-b)^\top u)\,\middle|\,
		\begin{aligned}
			&Ax+By\leq b,\,Cx+Dy\leq d,\\
			&B^\top u=-c,\,u\geq 0
		\end{aligned}
	\right\}
\end{equation}
in \cref{alg:boosted_DC_penalty}
where the appearing penalty term models the duality gap of the lower level problem.
This idea is basically taken from \cite{WhiteAnandalingam1993}. Observe that 
\eqref{eq:subproblem_duality_gap} is, in contrast to \eqref{eq:penalized_problem_DC}, a
fully explicit optimization problem which can be solved via standard methods from constrained
optimization. However, it does not possess the natural DC-structure we observed in
\eqref{eq:penalized_problem_DC}. Furthermore, we are in need to treat the lower level Lagrange
multiplier (or dual variable) $u$ as an explicit variable which might be a delicate issue since
\eqref{eq:subproblem_duality_gap} is related to a penalized version of the KKT-reformulation of
the bilevel optimization problem \eqref{eq:upper_level}.
More precisely, the presence of $u$ could induce artificial local minimizers and stationary
points of \eqref{eq:subproblem_duality_gap} which do not correspond to local minimizers or stationary points of
\eqref{eq:penalized_problem_DC}, see \cite{BenkoMehlitz2020,DempeDutta2012} for details on this phenomenon.

Here, we challenge these three penalty methods with the aid of three bilevel optimization problems:
\begin{enumerate}
	\item the fully linear bilevel optimization problem from \cite[Example~2]{BardFalk1982},
	\item the problem from \cite[Example~3.3]{LamparielloSagratella2017} which possesses a
		quadratic upper level objective function, and
	\item an inverse transportation problem where the offer has to be reconstructed from a noisy
		transportation plan. 
\end{enumerate}
For each of these examples, we first provide a description of the problem data. Afterwards, we present
our numerical results. More precisely, we challenge the three methods with random starting points
and compare the outcome by means of computed function values, number of (outer) penalty iterations,
number of (inner) DC-iterations (only for the DC-type methods), and the size of the (lower level)
duality gap at the computed solution.
In order to provide a reasonable visually convincing quantitative comparison, 
we make use of performance profiles, see \cite{DolanMore2002}.

In the reminder of this section, we first comment on the actual numerical implementation of the
algorithms. Afterwards, our results are presented.

\subsubsection{Implementation}

For a numerical comparison, we implement the following penalty methods for the computational
treatment of \eqref{eq:upper_level}:
\begin{itemize}[leftmargin=1.7cm]
	\item[\textbf{PBDC:}] \cref{alg:boosted_DC_penalty} where the subproblems 
		\eqref{eq:penalized_problem_DC} are solved
		with the aid the boosted DC-method from \cite{AragonArtachoCampoyVuong2020},
	\item[\textbf{PDC:}] \cref{alg:boosted_DC_penalty} where the subproblems
		\eqref{eq:penalized_problem_DC} are solved with the standard DC-method, and
	\item[\textbf{PDG:}] in each iteration, we solve the subproblem
		\eqref{eq:subproblem_duality_gap} instead of \eqref{eq:penalized_problem_DC}
		in \cref{alg:boosted_DC_penalty}.
\end{itemize}
All these methods have been implemented using MATLAB 2021a. 
For the solution of the appearing subproblems, we employed MATLAB's \texttt{fmincon} in default
mode. Furthermore, appearing linear optimization problems, e.g., for the pointwise evaluation of the
value function $\vartheta$ or its subdifferential, see \cref{lem:subdiff_of_theta}, are
solved via MATLAB's \texttt{linprog} routine. The sequence of penalty parameters is generated 
via $\sigma^0:=1$ and $\gamma:=1.2$.
Each of the algorithms is terminated whenever
we have $|c^\top(y^{k+1}-y^{k+1}_s)|\leq 10^{-7}$ for an arbitrary lower level
solution $y^{k+1}_s\in\Psi(x^{k+1})$ or the number of outer iterations exceeds $200$ 
(the latter, actually, did not happen).
Note that this weakens the actual termination criterion $y^{k+1}\in\Psi(x^{k+1})$ which was
used in \cref{alg:boosted_DC_penalty}.
For \textbf{PBDC} and \textbf{PDC}, we limited the number of inner DC-iterations to $100$
(we never hit this bound during our experiments).
Furthermore, using the notation from \cref{sec:theory_DC}, we exploited the standard 
termination criterion $\norm{z^k-w^k}\leq 10^{-4}$ for
both DC-methods. According to \cite{AragonArtachoCampoyVuong2020}, the parameters for the
line search in the boosted DC-method are fixed to $\bar\lambda:=1$, $\alpha:=10^{-2}$, and
$\beta:=10^{-1}$. In order to enhance the numerical performance of the DC-methods
\textbf{PBDC} and \textbf{PDC}, we added the zero 
$\tfrac12(\norm{x}^2+\norm{y}^2)-\tfrac12(\norm{x}^2+\norm{y}^2)$
to the objective function of \eqref{eq:penalized_problem_DC} in order to make both convex parts
of it strongly convex.

All methods were challenged by 100 random starting points. These were generated by firstly choosing
random vectors from a suitable box which are secondly projected onto the polyhedron 
$\mathcal Z_u\cap\mathcal Z_\ell$.
For the method \textbf{PDG},
we additionally chose $u^0$ as a solution of the linear optimization problem
$\min_u\{\mathtt e^\top u\,|\,B^\top u=-c,\,u\geq 0\}$, where $\mathtt e\in\R^p$ denotes
the all-ones-vector, and, for each $k\in\N$, used
$(x^k,y^k,u^k)$ as the initial point for the numerical solution of 
the subproblem \eqref{eq:subproblem_duality_gap}.

Finally, let us briefly comment on the creation of performance profiles.
For the set $\mathcal A:=\{\textbf{PBDC},\textbf{PDC},\textbf{PDG}\}$ of algorithms
and the set $\mathcal S:=\{1,\ldots,\ell\}$ of indices associated with random starting
points, let $w^s_a$ be the output of algorithm $a\in\mathcal A$ with random starting point 
$s\in\mathcal S$. For a scalar performance measure $\pi$ (representing computed function values,
number of outer iterations, final lower level duality gap, or number of inner iterations), 
we consider the performance metric given by
\[
	\mathcal Q_\theta^\pi(w_a^s)
	:=
	\begin{cases}
		\pi(w^s_a)-\pi^*+\theta&\text{if }w^s_a\text{ satisfies the termination criterion},\\
		\infty	&\text{otherwise.}
	\end{cases}
\]
Here, $\pi^*$ represents a benchmark value which is chosen to be the globally optimal function value (or a
suitable approximate of it) if $\pi$ measures computed function values, and simply zero in the
other three cases. Furthermore, $\theta\geq 0$ is an additional parameter which reduces sensitivity to
numerical accuracy and will be specified in the respective experiments. 
For the performance ratio
\[
	\forall a\in\mathcal A\,\forall s\in\mathcal S\colon\quad
	r_{a,s}^\pi
	:=
	\frac{\mathcal Q_\theta^\pi(w_a^s)}
	{\min\{\mathcal Q_\theta^\pi(w_{a'}^s)\,|\,a'\in\mathcal A\}},
\]
given for fixed $\pi$,
we plot the illustrative parts of the curves $\rho_a^\pi\colon[1,\infty)\to[0,1]$ ($a\in\mathcal A$), 
defined by
\[
	\forall\tau\in[1,\infty)\colon\quad
	\rho_a^\pi(\tau):=
	\frac{\card\{s\in\mathcal S\,|\,r_{a,s}^\pi\leq\tau\}}{\card(\mathcal S)}
\]
where $\card(\cdot)$ assigns to each input set its cardinality.

\subsubsection{Numerical experiments}\label{sec:experiments}

\paragraph*{Experiment 1}

We investigate the linear bilevel optimization problem
\begin{equation}\label{eq:BardFalk1982}\tag{Ex1}
	\min\limits_{x,y}\{-2x_1+x_2+\tfrac12y_1\,|\,x\geq 0,\,y\in\Psi(x)\},
\end{equation}
where $\Psi\colon\R^2\tto\R^2$ is given by
\[
	\forall x\in\R^2\colon\quad
	\Psi(x)
	:=
	\argmin\limits_y
	\left\{-4y_1+y_2\,\middle|\,
		\begin{aligned}&2x_1-y_1+y_2\geq \tfrac52,\,x_1+x_2\leq 2,\\
		&x_1-3x_2+y_2\leq 2,\,y\geq 0
		\end{aligned}
	\right\},
\]
which is taken from \cite[Example~2]{BardFalk1982}.
For the creation of random starting points, we made use of the box $[0,2]^4$.
The optimal objective value of this program is given by $f^*:=-3.25$.
The resulting performance profiles as well as the chosen offset parameters $\theta$ can be
found in \cref{fig:PerfProf_Ex1}. Additionally, some averaged numbers are presented in  
\cref{tab:BardFalk1982}.
\begin{figure}[ht]\centering
\includegraphics[width=7.0cm]{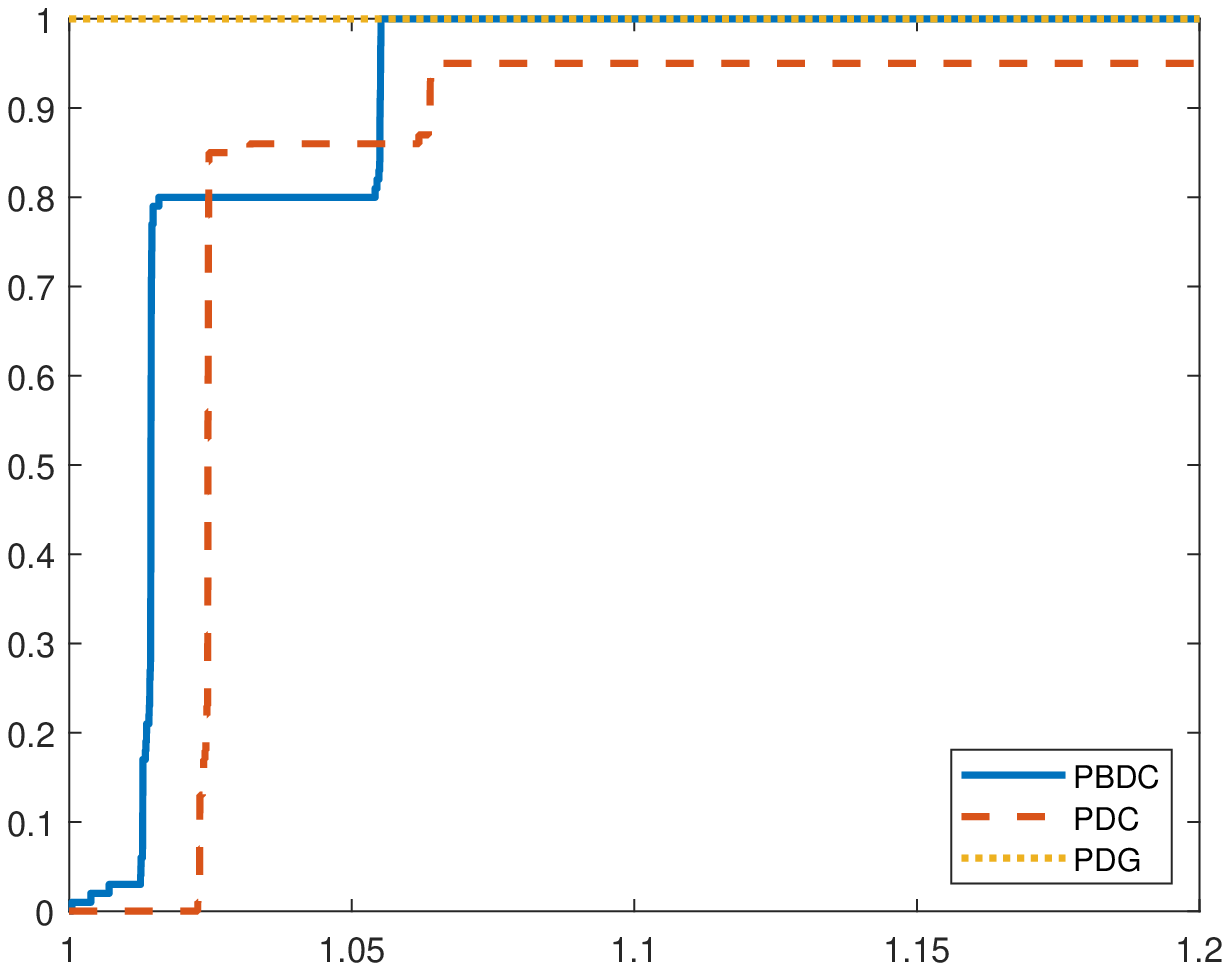}
\includegraphics[width=7.0cm]{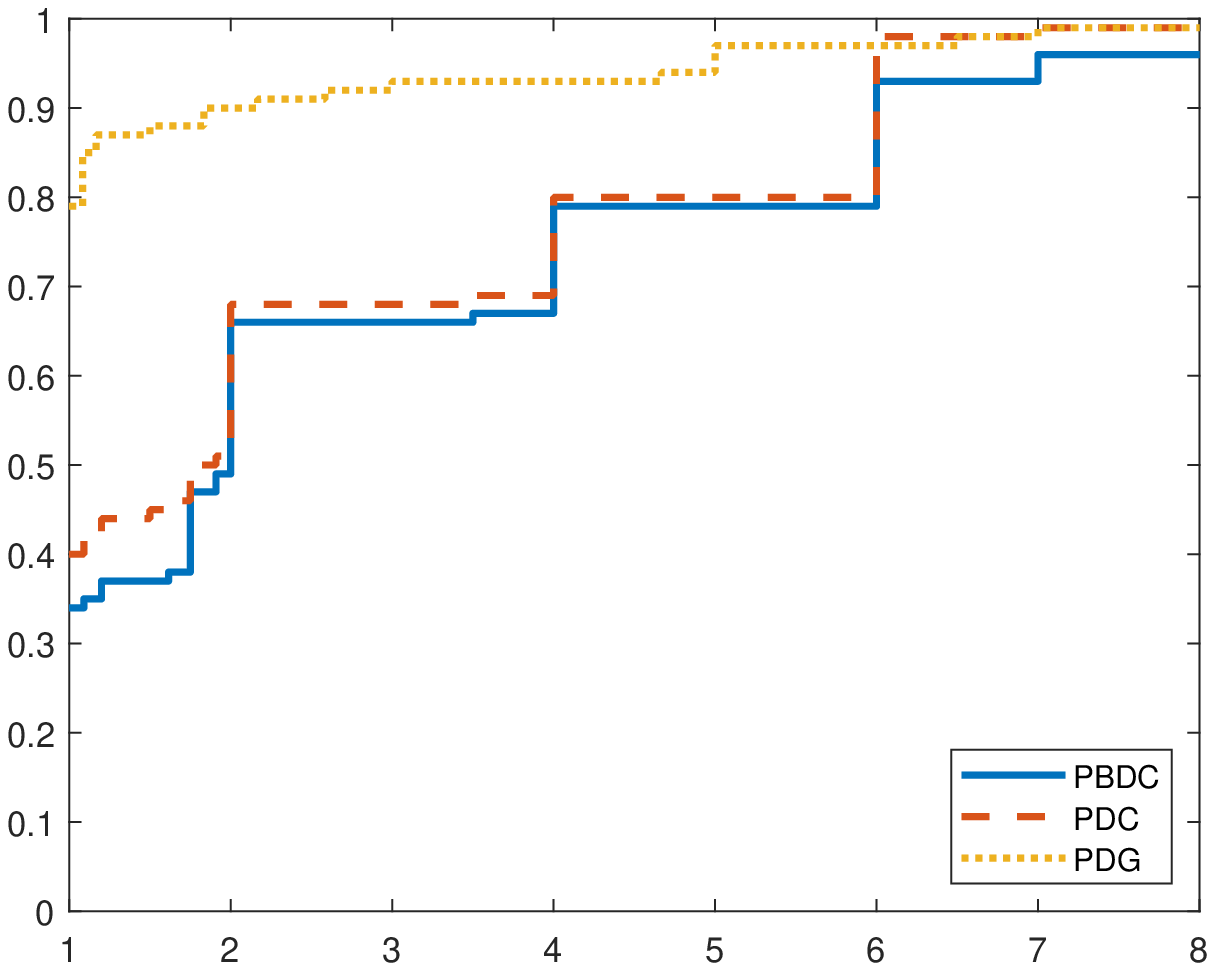}
\\
\includegraphics[width=7.0cm]{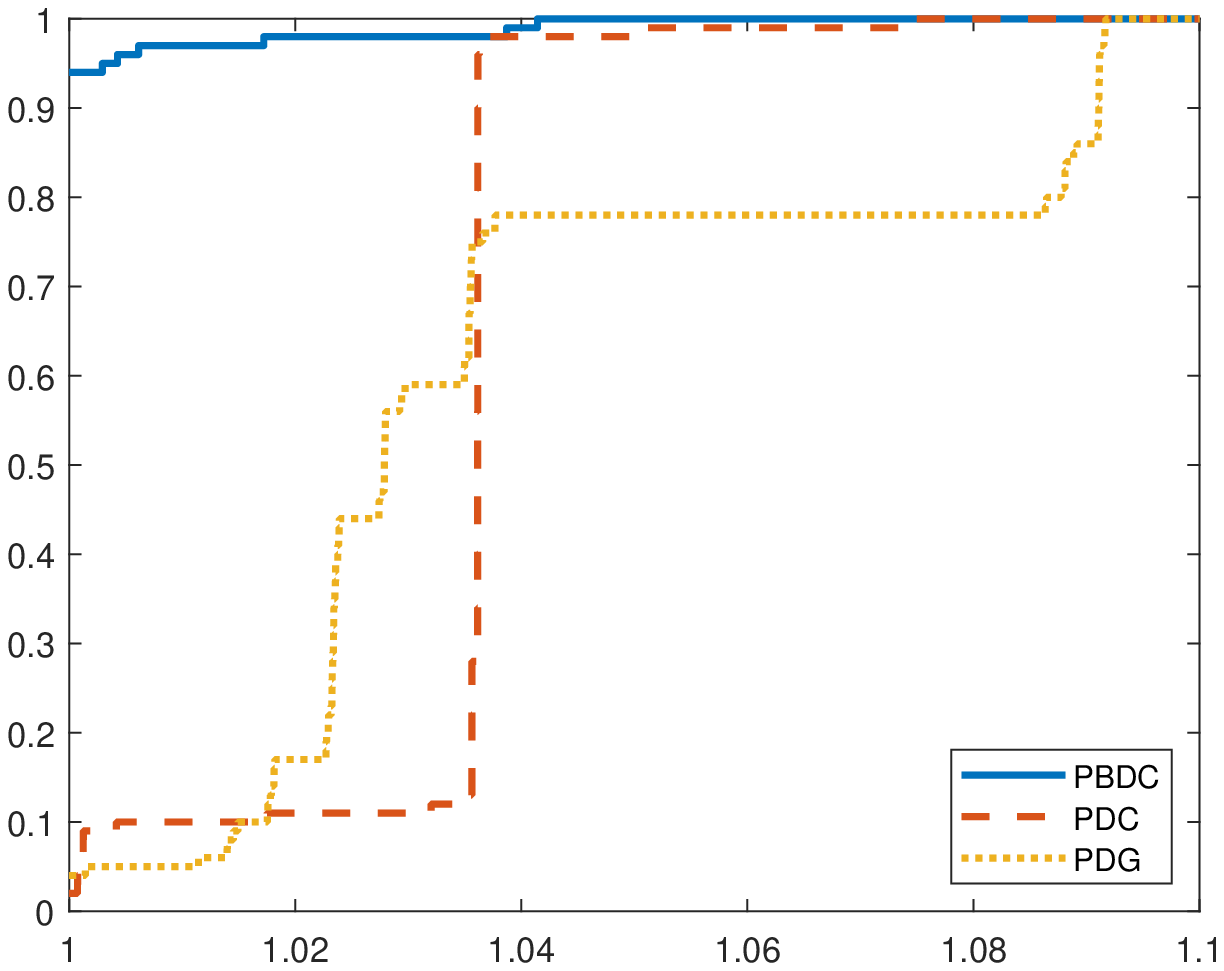}
\includegraphics[width=7.0cm]{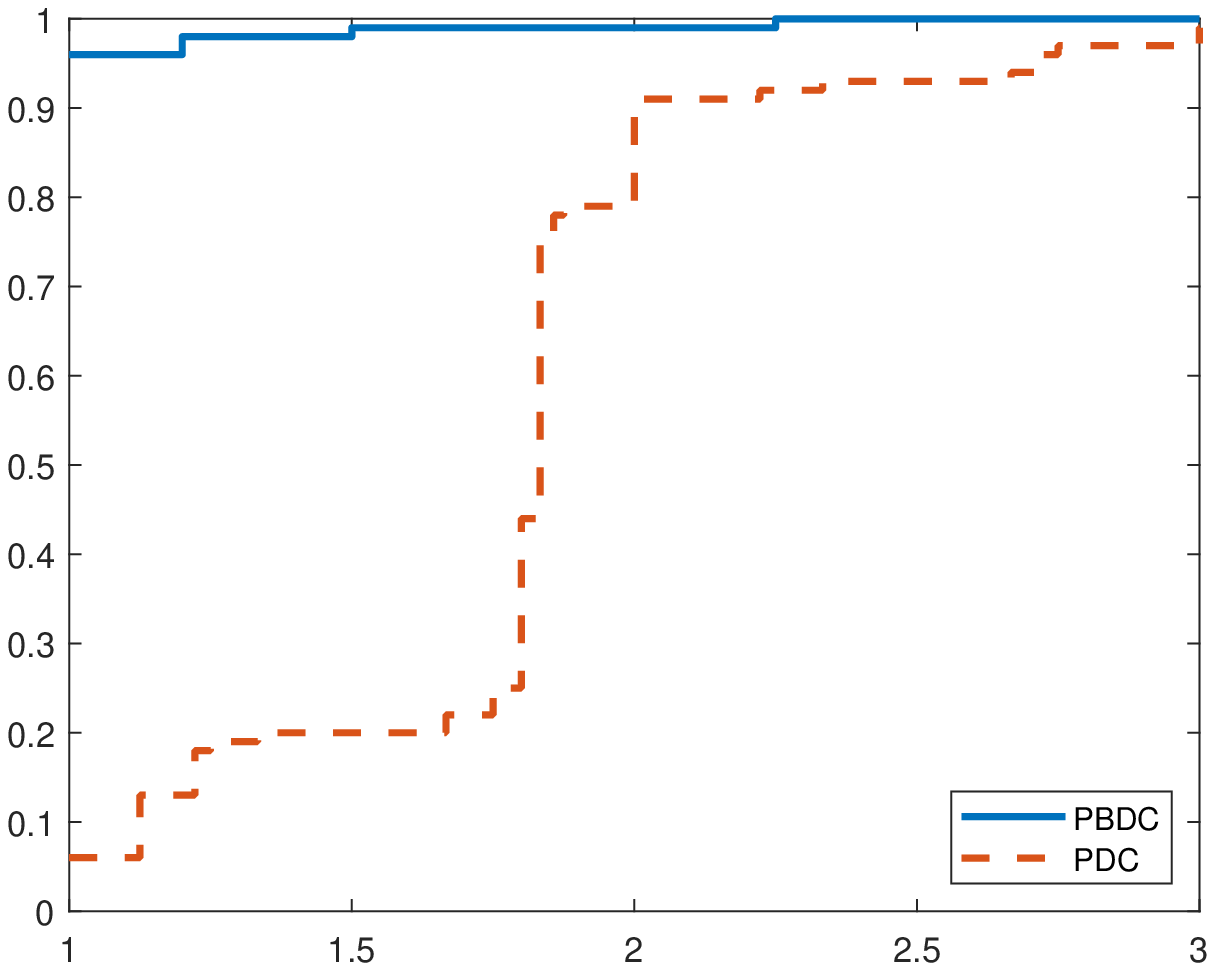}
\caption{Performance profiles for \eqref{eq:BardFalk1982}.
From top left to bottom right: function values ($\theta=10^{-4}$), outer iterations ($\theta=1$),
	duality gap ($\theta=10^{-6}$), and total inner iterations ($\theta=1$).}
\label{fig:PerfProf_Ex1}
\end{figure}
\begin{table}
\begin{center}
\begin{tabular}{l||l|l|l}
	&\textbf{PBDC}&\textbf{PDC}&\textbf{PDG}\\
	\hline
	average function value&-3.2500&-3.2214&-3.2500\\
	\hline
	average number of outer iterations&12.4000&11.2500&7.8600\\
	\hline
	average lower level duality gap&8.9684$\cdot10^{-9}$&4.0932$\cdot10^{-8}$&4.7177$\cdot10^{-8}$\\
	\hline
	average number of inner iterations&5.1600&9.6700&-
\end{tabular}
\caption{Averaged performance indices for \eqref{eq:BardFalk1982}.}
\label{tab:BardFalk1982}
\end{center}
\end{table}

It turns out that \textbf{PBDC} as well as \textbf{PDG} reliably compute the global minimizer of
\eqref{eq:BardFalk1982}, and both methods do not outrun \textbf{PDC} in this regard. 
However, we see that \textbf{PDG} needs less outer iterations than the DC-methods until 
the termination criterion is
reached. Nevertheless, we observe that \textbf{PBDC} needs less DC-iterations than \textbf{PDC}.
Interestingly, the outputs of \textbf{PBDC} come along with a duality gap which
is smaller by factor $10^{-1}$ than the upper bound $10^{-7}$ appearing in the 
termination criterion of the outer loop in several cases. 

\paragraph*{Experiment 2}

Next, we investigate \cite[Example~3.3]{LamparielloSagratella2017} which is given by
\begin{equation}\label{eq:LamparielloSagratella2017}\tag{Ex2}
	\min\limits_{x,y}\{x^2+(y_1+y_2)^2\,|\,x\geq 0.5,\,y\in\Psi(x)\}
\end{equation}
where $\Psi\colon\R\tto\R^2$ is defined by
\[
	\forall x\in\R\colon\quad
	\Psi(x):=\argmin\limits_y\{y_1\,|\,x+y_1+y_2\geq 1,\,y\geq 0\}.
\]
For the creation of random starting points, we exploited the box $[0,2]^3$.
The optimal objective value of this program is given by $f^*:=0.5$.
The resulting performance profiles as well as some averages regarding the performance
indices can be found in \cref{fig:PerfProf_Ex2} and \cref{tab:LamparielloSagratella},
respectively.
\begin{figure}[ht]\centering
\includegraphics[width=7.0cm]{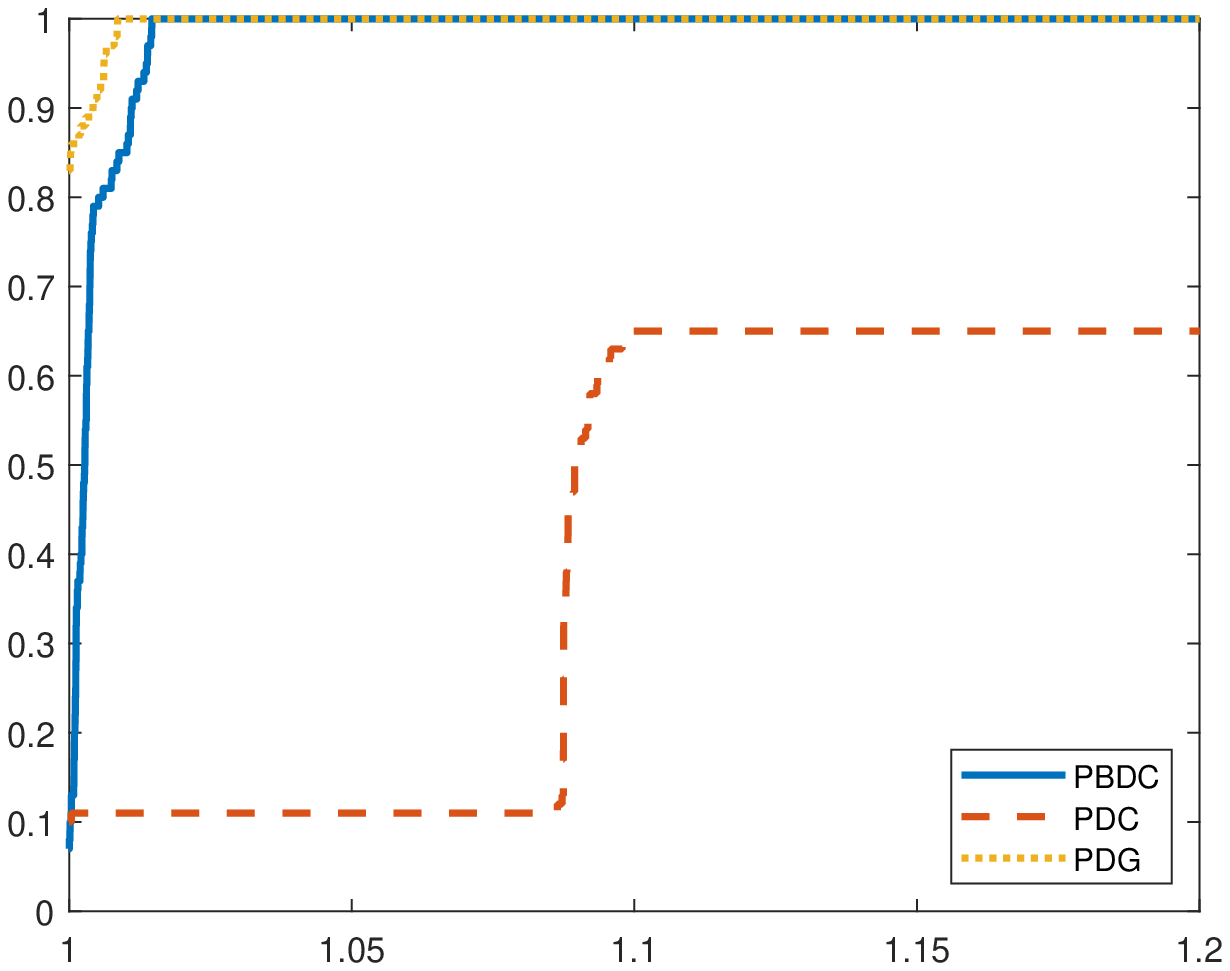}
\includegraphics[width=7.0cm]{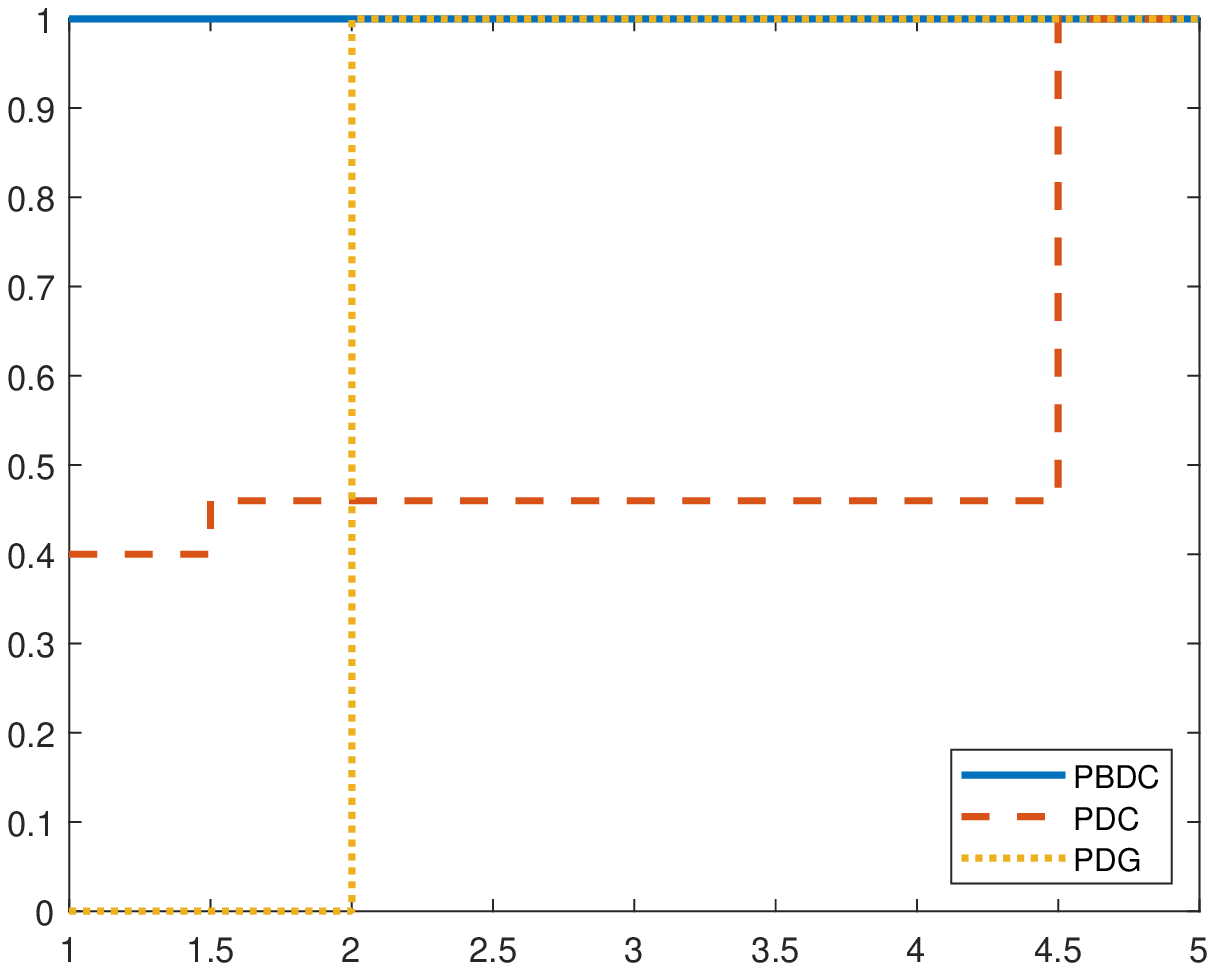}
\\
\includegraphics[width=7.0cm]{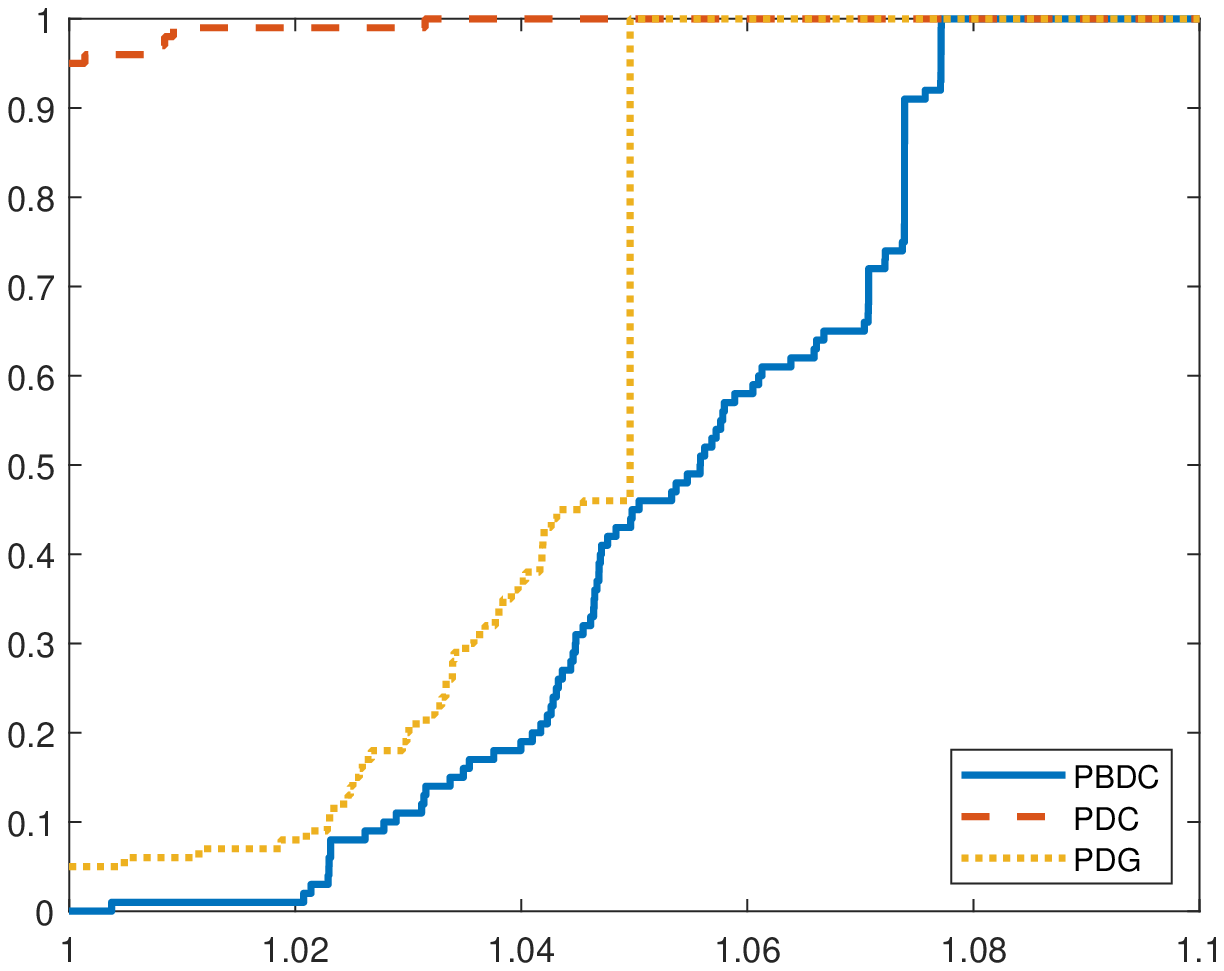}
\includegraphics[width=7.0cm]{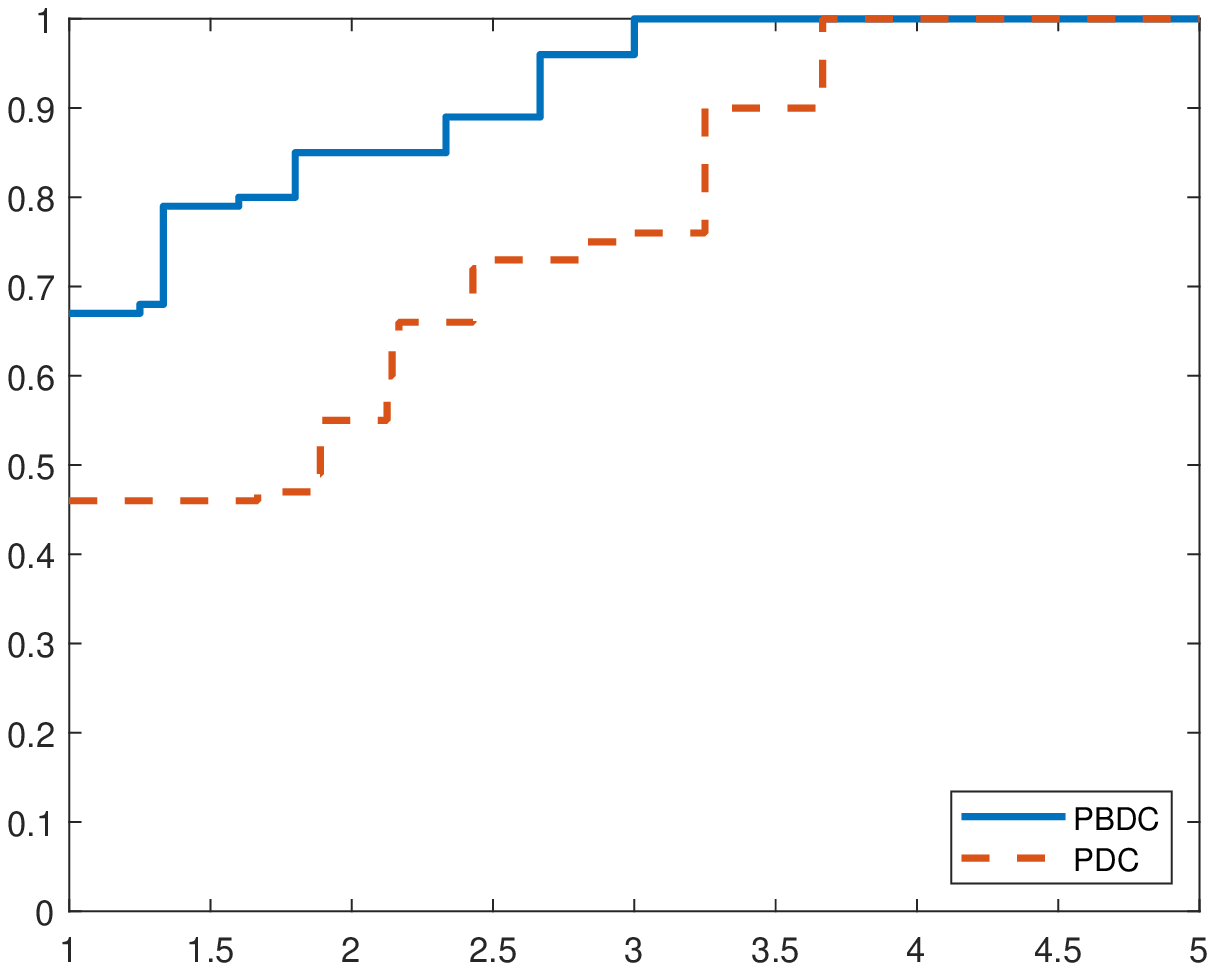}
\caption{Performance profiles for \eqref{eq:LamparielloSagratella2017}.
From top left to bottom right: function values ($\theta=10^{-5}$), outer iterations ($\theta=1$),
	duality gap ($\theta=10^{-6}$), and total inner iterations ($\theta=1$).}
\label{fig:PerfProf_Ex2}
\end{figure}
\begin{table}
\begin{center}
\begin{tabular}{l||l|l|l}
	&\textbf{PBDC}&\textbf{PDC}&\textbf{PDG}\\
	\hline
	average function value&0.5000&0.5285&0.5000\\
	\hline
	average number of outer iterations&1.0000&4.8400&3.0000\\
	\hline
	average lower level duality gap&7.0489$\cdot10^{-8}$&1.5821$\cdot10^{-8}$&5.5576$\cdot10^{-8}$\\
	\hline
	average number of inner iterations&4.7100&8.2700&-
\end{tabular}
\caption{Averaged performance indices for \eqref{eq:LamparielloSagratella2017}.}
\label{tab:LamparielloSagratella}
\end{center}
\end{table}

Similar to our first experiment, we observe that \textbf{PBDC} and \textbf{PDG} reliably compute
the global minimizer. For that purpose, \textbf{PBDC} only needs to run one outer penalty iteration
while \textbf{PDG} runs three outer iterations for each starting point. Regarding both criteria,
\textbf{PDC} shows some flaws. Inspecting the total number of inner iterations, \textbf{PBDC} behaves
much better than \textbf{PDC}. Finally, a look at the final duality gap shows that all three methods
just fall below the upper bound $10^{-7}$ which was used in the termination criterion. 
Here, \textbf{PDC} seems to have slight advantages over the other two methods. However, it is already
outperformed by the other two methods regarding the more important performance measures.

\paragraph*{Experiment 3}

Finally, we are going to challenge the three penalty methods by means of the inverse transportation
problem
\begin{equation}\label{eq:InverseTransportation}\tag{Ex3}
	\min\limits_{x,y}\{\tfrac12\norm{y-y_\textup{o}}^2\,|\,
		x\geq 0,\,\mathtt e^\top x\geq\mathtt e^\top b^\textup{dem},\,y\in\Psi(x)
		\}
\end{equation}
where $\Psi\colon\R^n\tto\R^{n\times \ell}$ is the solution mapping of the 
parametric transportation problem
\begin{equation}\label{eq:transprotation}\tag{TR$(x)$}
	\min\limits_{y}\left\{
		\mathsmaller\sum\nolimits_{i=1}^n\mathsmaller\sum\nolimits_{j=1}^\ell c_{ij}y_{ij}\,\middle|
		\begin{aligned}
			&\mathsmaller\sum\nolimits_{j=1}^\ell y_{ij}\leq x_i\,(i=1,\ldots,n),\\
			&\mathsmaller\sum\nolimits_{i=1}^ny_{ij}\geq b^\textup{dem}_j\,(j=1,\ldots,\ell),\\
			&y\geq 0
		\end{aligned}
		\right\}.
\end{equation}
Above, $\ell\in\N$ is a positive integer, $b^\textup{dem}\in\{0,\ldots,10\}^\ell$ 
is a random integer vector which models the minimum demand of the $\ell$ consumers, and
$c\in[0,1]^{n\times \ell}$ is a randomly chosen cost matrix. In \eqref{eq:transprotation}, the parameter
$x\in\R^n$ represents the offer provided at the $n$ warehouses which is unknown and shall be 
reconstructed from a given (noised) transportation plan $y_\textup{o}\in\R^{n\times \ell}$.
The latter is constructed in the following way: 
For $x_\textup{d}:=(\mathtt e^\top b^\textup{dem})/n\,\mathtt e$, 
we choose $y_\textup{d}\in\Psi(x_\textup{d})$. Afterwards, some noise is added to $y_\textup{d}$ in
order to create $y_\textup{o}$. 

For our experiments, we chose $n:=5$ and $\ell:=7$.
The precise values of the data $c$, $b^\textup{dem}$, and $y_\textup{o}$ used for our experiments can be found in \cref{sec:appendix}. 
The components of the random starting points are chosen from the interval $[0,6]$. 
The minimum function value realized in our experiments is given by $f^*=5.000776\cdot 10^{-4}$.
The associated point $(x^*,y^*)$ is also given in \cref{sec:appendix}.
The resulting
performance profiles and averaged performance indices are documented in
\cref{fig:PerfProf_Ex3} and \cref{tab:InverseTransportation}, respectively.
\begin{figure}[ht]\centering
\includegraphics[width=7.0cm]{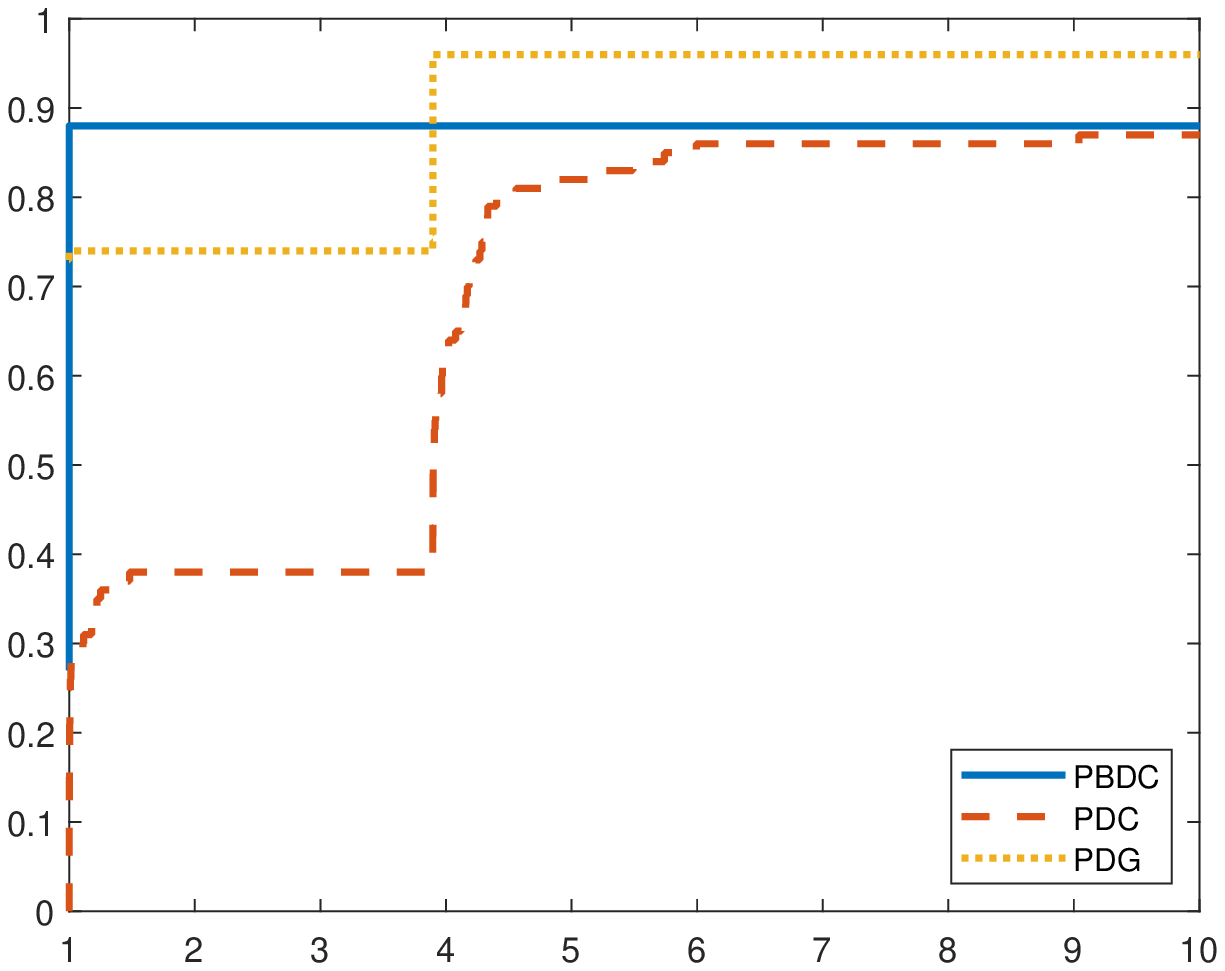}
\includegraphics[width=7.0cm]{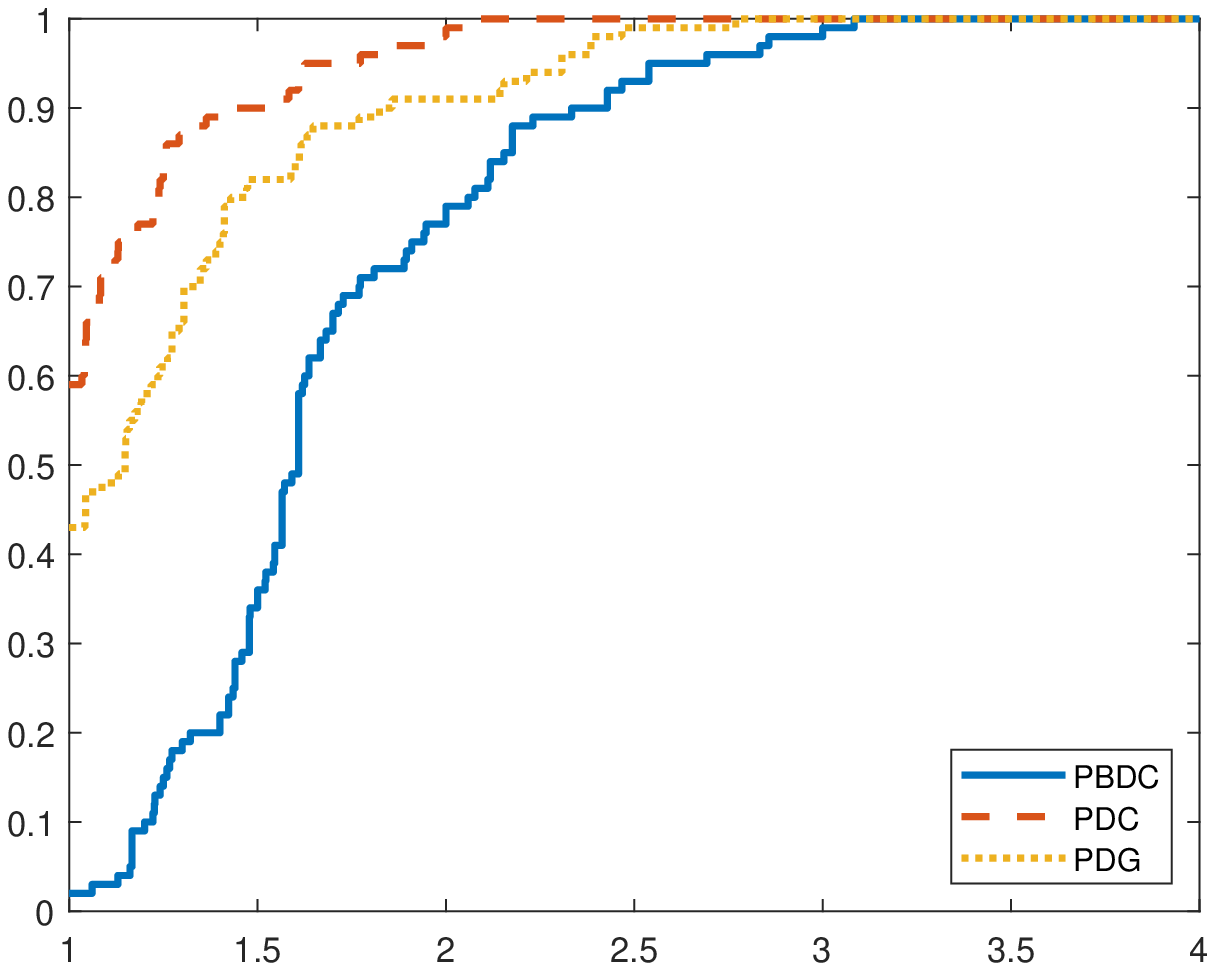}
\\
\includegraphics[width=7.0cm]{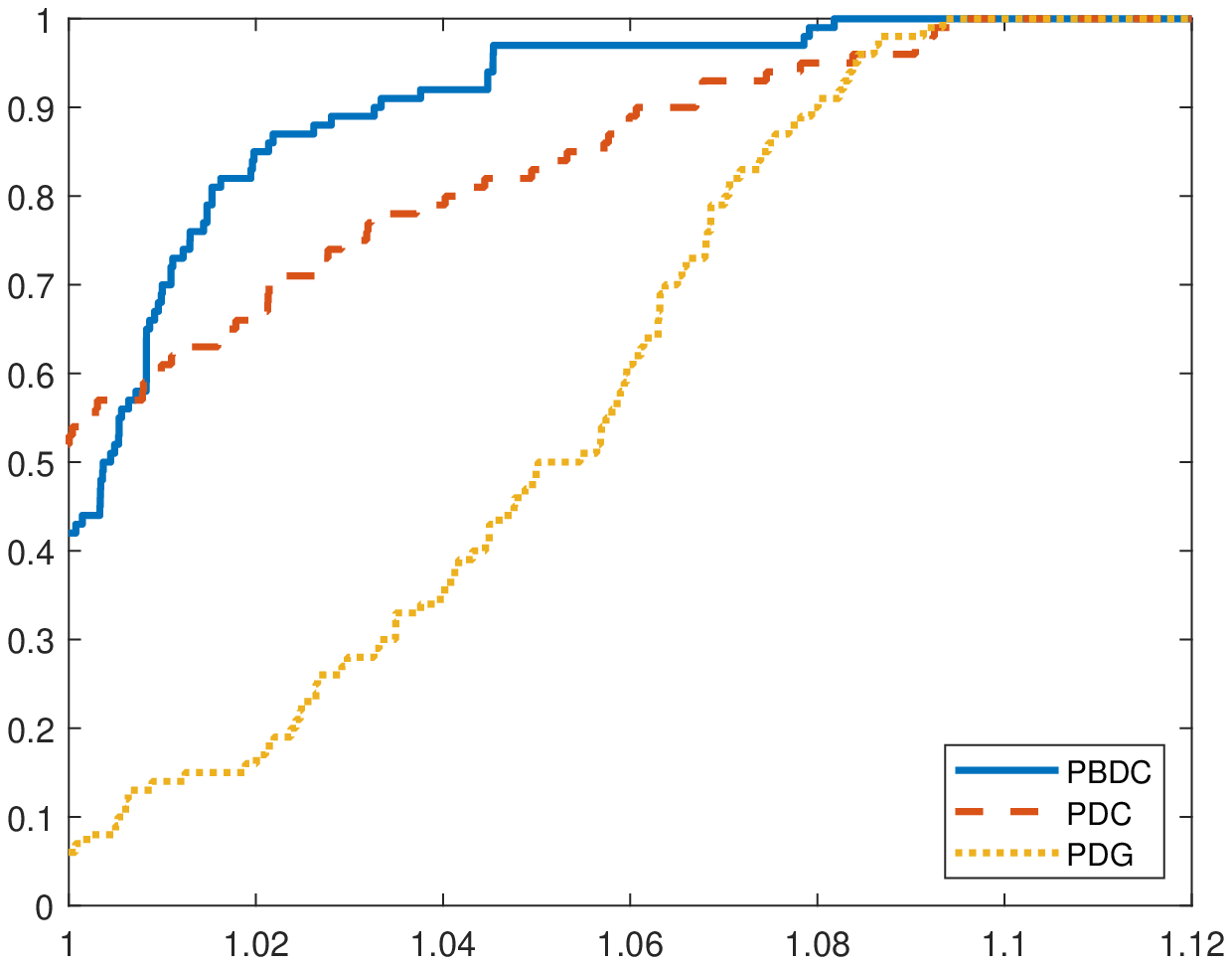}
\includegraphics[width=7.0cm]{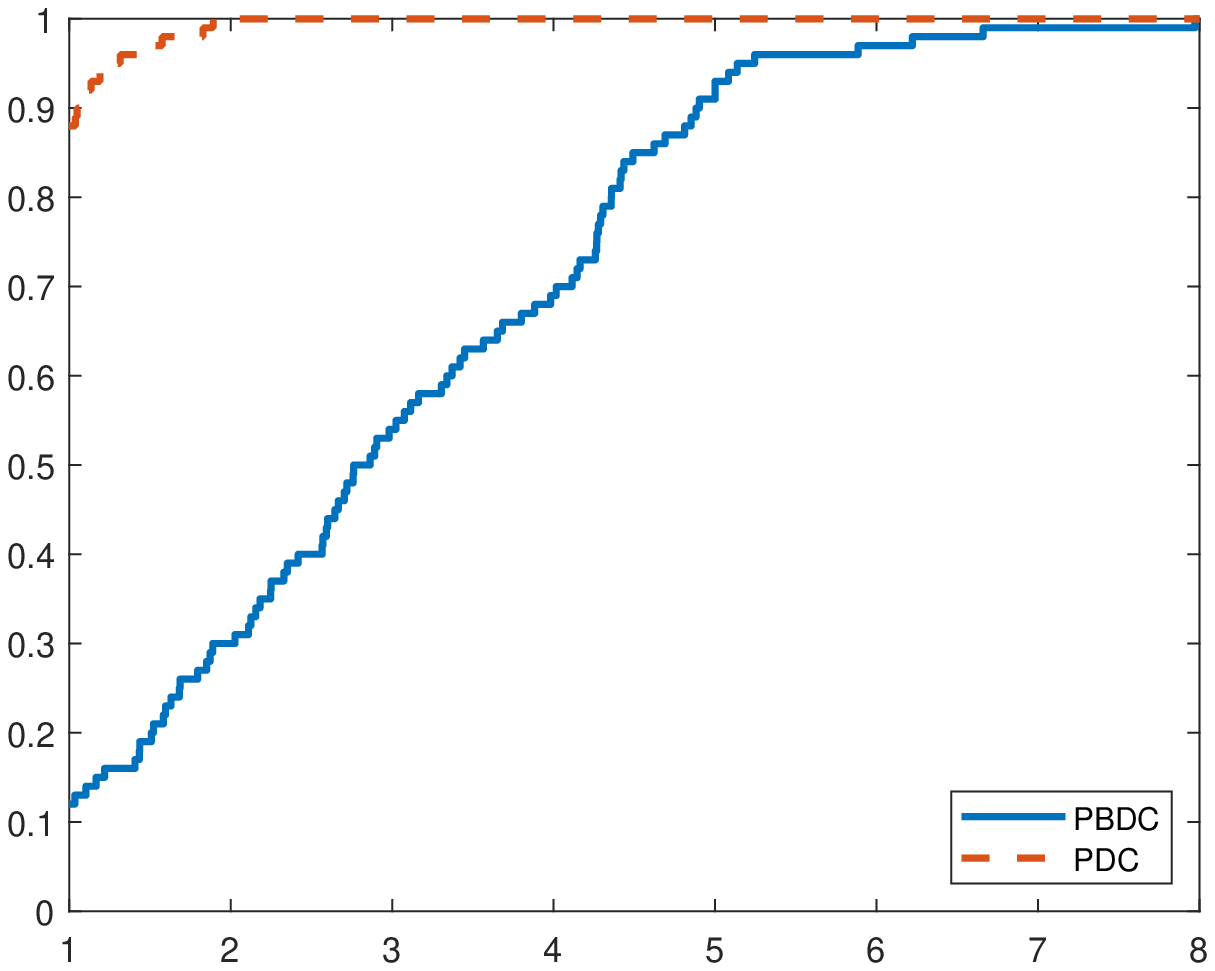}
\caption{Performance profiles for \eqref{eq:InverseTransportation}.
From top left to bottom right: function values ($\theta=10^{-2}$), outer iterations ($\theta=1$),
	duality gap ($\theta=10^{-6}$), and total inner iterations ($\theta=1$).}
\label{fig:PerfProf_Ex3}
\end{figure}

\begin{table}
\begin{center}
\begin{tabular}{l||l|l|l}
	&\textbf{PBDC}&\textbf{PDC}&\textbf{PDG}\\
	\hline
	average function value&0.5380&0.5561&0.2004\\
	\hline
	average number of outer iterations&33.7200&22.3600&25.4200\\
	\hline
	average lower level duality gap&2.4727$\cdot10^{-8}$&3.2808$\cdot10^{-8}$&6.2347$\cdot10^{-8}$\\
	\hline
	average number of inner iterations&177.2100&61.1100&-
\end{tabular}
\caption{Averaged performance indices for \eqref{eq:InverseTransportation}.}
\label{tab:InverseTransportation}
\end{center}
\end{table}

Clearly, \eqref{eq:InverseTransportation} is a far more challenging problem than 
\eqref{eq:BardFalk1982} or \eqref{eq:LamparielloSagratella2017}. Throughout the
test runs, we observed that the algorithms tend to identify different stationary
points of \eqref{eq:InverseTransportation} with heavily differing objective values.
However, in 87 of the 100 test runs, \textbf{PBDC} approximately identified $(x^*,y^*)$.
In this regard, \textbf{PDC} only succeeded in 10 of the test runs while
\textbf{PDG} was successful in 71 of the test runs. Let us point out that the first row
of \cref{tab:InverseTransportation} might be slightly misleading. Particularly,
the surprisingly large value for the average function value related to \textbf{PBDC}
results from the fact that the algorithm finds a point with objective value $2.9051$ in $7$
runs and a point with objective value $5.0909$ in $5$ runs.
The performance advantage of \textbf{PBDC} regarding function values comes for a price,
namely, a significantly larger number of outer and inner iterations in comparison with
the other two methods. Regarding the size of the final duality gap, we did not figure
out any surprising behavior. Let us underline that although \textbf{PDC}
comes along with the smallest number of outer and inner iterations, its outputs are often
far away from $(x^*,y^*)$ which is why this advantage seems to be practically irrelevant.

\subsubsection{Summary}

Throughout the experiments, we observed that \textbf{PBDC} computes reasonable points, i.e.,
global minimizers, of \eqref{eq:upper_level} in most of the test runs. For the smaller test
instances \eqref{eq:BardFalk1982} and \eqref{eq:LamparielloSagratella2017}, some performance
advantage regarding \textbf{PDC} w.r.t.\ inner and outer iteration numbers has been observed.
In the more challenging setting of \eqref{eq:InverseTransportation}, we abstain from putting
too much emphasis on the iteration numbers of \textbf{PDC} since this algorithm did not compute
points near the global minimizer in most of the runs. Furthermore, we attest \textbf{PDG} a
solid performance regarding computed function values and iteration numbers. Here, the simplicity
of the subproblem \eqref{eq:subproblem_duality_gap} seems to pay off particularly for smaller
problem instances.

\section{Conclusions and perspectives}\label{sec:conclusions}

In this paper, we applied the concepts of asymptotic stationarity and regularity to
nonlinear optimization problems with potentially nonsmooth but Lipschitzian data functions. 
Our theoretical investigations led to the formulation of three comparatively weak 
regularity conditions which
enrich the landscape of available constraint qualifications in the field of nonsmooth programming,
see \cref{def:asymptotic_regularity} and \cref{fig:CQs}.
Afterwards, we investigated complementarity-constrained programs in order to show that these
quite general concepts possess some reasonable extensions to disjunctive programs where they can be
used in order to carry out the convergence analysis associated with
some solution methods under weak assumptions.
Pointing the reader's attention back to \cref{rem:or_constrained_programming}, our results make
clear that similar concepts can be easily obtained for or- and vanishing-constrained optimization
problems. It remains a task for future research to study the capability of asymptotic stationarity
and regularity in the context of the numerical treatment of these problem classes.
Finally, we exploited asymptotic regularity in the context of bilevel optimization.
More precisely, we justified certain stationarity conditions as well as a penalty method 
for the numerical solution
of affinely constrained bilevel optimization problems. 
Results of some computational experiments are shown in order to provide a quantitative justification
of our approach. We already pointed out that the overall
theory can be easily extended to bilevel optimization problems with nonlinear but fully convex
constraints. However, in this situation, asymptotic regularity is no longer inherently satisfied.
It remains to be seen whether this concept yields applicable constraint qualifications for 
more general bilevel programming problems. Furthermore, it has to be studied whether the ideas behind 
\cref{alg:boosted_DC_penalty} still lead to convincing numerical results as soon as further
nonlinearities appear in the problem data. 
Suitable test problems can be found in the BOLIB collection from \cite{ZhouZemkohoTin2020} which
also comprises \eqref{eq:BardFalk1982} and \eqref{eq:LamparielloSagratella2017} considered in
\cref{sec:experiments}.


\appendix
\section{Appendix}\label{sec:appendix}

Here, we provide the missing data for Experiment 3 from \cref{sec:experiments}.
First, we state the data matrices $c\in[0,1]^{5\times 7}$, $b^\textup{dem}\in\{1,\ldots,10\}^7$,
and $y_\textup{o}\in\R^{5\times 7}$.
\begin{align*}
	c
	&=
	\begin{pmatrix}
		0.5757&0.8423&0.4997&0.4390&0.1491&0.0283&0.7567\\
		0.7961&0.2936&0.1152&0.3751&0.8289&0.8418&0.6652\\
		0.9601&0.9431&0.1127&0.6483&0.4808&0.0665&0.8978\\
		0.4972&0.7713&0.0604&0.2625&0.6511&.01336&0.6385\\
		0.3849&0.7657&0.6529&0.3815&0.0300&0.3401&0.9189
	\end{pmatrix}\\
	b^\textup{dem}
	&=
	\begin{pmatrix}
		5&5&5&10&3&9&1
	\end{pmatrix}^\top\\
	y_\textup{o}
	&=
	\begin{pmatrix}
		-0.0032&0.0053&-0.0031&0.0024&2.9991&4.5902&0.0020\\
		0.0020&5.0030&1.5969&-0.0001&0.0040&0.0078&0.9911\\
		-0.0080&0.0030&3.2053&0.0098&-0.0075&4.3973&0.0035\\
		-0.0025&0.0073&0.1958&7.3927&0.0035&-0.0059&0.0074\\
		5.0050&-0.0016&-0.0100&2.5930&-0.0045&0.0074&0.0020
	\end{pmatrix}
\end{align*}
Next, we state the solution $(x^*,y^*)\in\R^5\times\R^{5\times 7}$ 
with the best function value $f^*=5.000766\cdot 10^{-4}$
found during our experiments. 
\begin{align*}
	x^*
	&=
	\begin{pmatrix}
		7.5965&7.5975&7.6095&7.5964&7.6002
	\end{pmatrix}^\top
	\\
	y^*
	&=
	\begin{pmatrix}
		0&0&0&0&3.0000&4.5965&0\\
		0&5.0000&1.5975&0&0&0&1.0000\\
		0&0&3.2060&0&0&4.4035&0\\
		0&0&0.1965&7.3998&0&0&0\\
		5.0000&0&0&2.6002&0&0&0
	\end{pmatrix}
\end{align*}
We note that the \emph{desired} pair of variables 
$(x_\textup{d},y_\textup{d})\in\R^5\times\R^{5\times 7}$,
which has been used for the precise construction of the problem data, is given as stated below.
\begin{align*}
	x_\textup{d}
	&=
	\begin{pmatrix}
		7.6000&7.6000&7.6000&7.6000&7.6000
	\end{pmatrix}^\top\\
	y_\textup{d}
	&=
	\begin{pmatrix}
		0&0&0&0&3.0000&4.6000&0\\
		0&5.0000&1.6000&0&0&0&1.0000\\
		0&0&3.2000&0&0&4.4000&0\\
		0&0&0.2000&7.4000&0&0&0\\
		5.0000&0&0&2.6000&0&0&0
	\end{pmatrix}
\end{align*}

\end{document}